\documentclass[10pt,leqno]{amsart}
\usepackage{amssymb}
\usepackage{xypic}
\usepackage{graphics}
\usepackage{array, graphicx}
\usepackage{tikz}
\usepackage{tikz-cd}
\usepackage [english]{babel}
\usepackage [autostyle, english = american]{csquotes}
\MakeOuterQuote{"}
\begin{document}
\theoremstyle{plain}
\newtheorem{thm}{Theorem}[section]
\newtheorem*{thm1}{Theorem 1}
\newtheorem*{thm2}{Theorem 2}
\newtheorem{lemma}[thm]{Lemma}
\newtheorem{lem}[thm]{Lemma}
\newtheorem{cor}[thm]{Corollary}
\newtheorem{prop}[thm]{Proposition}
\newtheorem{propose}[thm]{Proposition}
\newtheorem{variant}[thm]{Variant}
\newtheorem{conjecture}[thm]{Conjecture}
\theoremstyle{definition}
\newtheorem{notations}[thm]{Notations}
\newtheorem{rem}[thm]{Remark}  
\newtheorem{rmk}[thm]{Remark}
\newtheorem{rmks}[thm]{Remarks}
\newtheorem{defn}[thm]{Definition}
\newtheorem{ex}[thm]{Example}
\newtheorem{claim}[thm]{Claim}
\newtheorem{ass}[thm]{Assumption} 
\numberwithin{equation}{section}
\newcounter{elno}                \newcommand{\mc}{\mathcal}
\newcommand{\mb}{\mathbb}
\newcommand{\surj}{\twoheadrightarrow}
\newcommand{\inj}{\hookrightarrow} \newcommand{\zar}{{\rm zar}}
\newcommand{\an}{{\rm an}} \newcommand{\red}{{\rm red}}
\newcommand{\Rank}{{\rm rk}} \newcommand{\codim}{{\rm codim}}
\newcommand{\rank}{{\rm rank}} \newcommand{\Ker}{{\rm Ker \ }}
\newcommand{\Pic}{{\rm Pic}} \newcommand{\Div}{{\rm Div}}
\newcommand{\Hom}{{\rm Hom}} \newcommand{\im}{{\rm im}}
\newcommand{\Spec}{{\rm Spec \,}} \newcommand{\Sing}{{\rm Sing}}
\newcommand{\sing}{{\rm sing}} \newcommand{\reg}{{\rm reg}}
\newcommand{\Char}{{\rm char}} \newcommand{\Tr}{{\rm Tr}}
\newcommand{\Gal}{{\rm Gal}} \newcommand{\Min}{{\rm Min \ }}
\newcommand{\Max}{{\rm Max \ }} \newcommand{\Alb}{{\rm Alb}\,}
\newcommand{\GL}{{\rm GL}\,} 
\newcommand{\ie}{{\it i.e.\/},\ } \newcommand{\niso}{\not\cong}
\newcommand{\nin}{\not\in}
\newcommand{\soplus}[1]{\stackrel{#1}{\oplus}}
\newcommand{\by}[1]{\stackrel{#1}{\rightarrow}}

\newcommand{\longby}[1]{\stackrel{#1}{\longrightarrow}}
\newcommand{\vlongby}[1]{\stackrel{#1}{\mbox{\large{$\longrightarrow$}}}}
\newcommand{\ldownarrow}{\mbox{\Large{\Large{$\downarrow$}}}}
\newcommand{\lsearrow}{\mbox{\Large{$\searrow$}}}
\renewcommand{\d}{\stackrel{\mbox{\scriptsize{$\bullet$}}}{}}
\newcommand{\dlog}{{\rm dlog}\,} 
\newcommand{\longto}{\longrightarrow}
\newcommand{\vlongto}{\mbox{{\Large{$\longto$}}}}
\newcommand{\limdir}[1]{{\displaystyle{\mathop{\rm lim}_{\buildrel\longrightarrow\over{#1}}}}\,}
\newcommand{\liminv}[1]{{\displaystyle{\mathop{\rm lim}_{\buildrel\longleftarrow\over{#1}}}}\,}
\newcommand{\norm}[1]{\mbox{$\parallel{#1}\parallel$}}
\newcommand{\into}{\hookrightarrow} \newcommand{\image}{{\rm image}\,}
\newcommand{\Lie}{{\rm Lie}\,} 
\newcommand{\CM}{\rm CM}
\newcommand{\sext}{\mbox{${\mathcal E}xt\,$}} 
\newcommand{\shom}{\mbox{${\mathcal H}om\,$}} 
\newcommand{\coker}{{\rm coker}\,} 
\newcommand{\sm}{{\rm sm}}
\newcommand{\tensor}{\otimes}
\renewcommand{\iff}{\mbox{ $\Longleftrightarrow$ }}
\newcommand{\supp}{{\rm supp}\,}
\newcommand{\ext}[1]{\stackrel{#1}{\wedge}}
\newcommand{\onto}{\mbox{$\,\>>>\hspace{-.5cm}\to\hspace{.15cm}$}}
\newcommand{\propsubset} {\mbox{$\textstyle{
\subseteq_{\kern-5pt\raise-1pt\hbox{\mbox{\tiny{$/$}}}}}$}}
\newcommand{\sA}{{\mathcal A}}
\newcommand{\sB}{{\mathcal B}} \newcommand{\sC}{{\mathcal C}}
\newcommand{\sD}{{\mathcal D}} \newcommand{\sE}{{\mathcal E}}
\newcommand{\sF}{{\mathcal F}} \newcommand{\sG}{{\mathcal G}}
\newcommand{\sH}{{\mathcal H}} \newcommand{\sI}{{\mathcal I}}
\newcommand{\sJ}{{\mathcal J}} \newcommand{\sK}{{\mathcal K}}
\newcommand{\sL}{{\mathcal L}} \newcommand{\sM}{{\mathcal M}}

\newcommand{\sN}{{\mathcal N}} \newcommand{\sO}{{\mathcal O}}
\newcommand{\sP}{{\mathcal P}} \newcommand{\sQ}{{\mathcal Q}}
\newcommand{\sR}{{\mathcal R}} \newcommand{\sS}{{\mathcal S}}
\newcommand{\sT}{{\mathcal T}} \newcommand{\sU}{{\mathcal U}}
\newcommand{\sV}{{\mathcal V}} \newcommand{\sW}{{\mathcal W}}
\newcommand{\sX}{{\mathcal X}} \newcommand{\sY}{{\mathcal Y}}
\newcommand{\sZ}{{\mathcal Z}} \newcommand{\ccL}{\sL}
 \newcommand{\A}{{\mathbb A}} \newcommand{\B}{{\mathbb
B}} \newcommand{\C}{{\mathbb C}} \newcommand{\D}{{\mathbb D}}
\newcommand{\E}{{\mathbb E}} \newcommand{\F}{{\mathbb F}}
\newcommand{\G}{{\mathbb G}} \newcommand{\HH}{{\mathbb H}}
\newcommand{\I}{{\mathbb I}} \newcommand{\J}{{\mathbb J}}
\newcommand{\M}{{\mathbb M}} \newcommand{\N}{{\mathbb N}}
\renewcommand{\P}{{\mathbb P}} \newcommand{\Q}{{\mathbb Q}}

\newcommand{\R}{{\mathbb R}} \newcommand{\T}{{\mathbb T}}
\newcommand{\Z}{{\mathbb Z}}
\title[Hilbert-Kunz density function and  asymptotic Hilbert-Kunz multiplicity]
{Hilbert-Kunz density function and  asymptotic Hilbert-Kunz multiplicity 
for projective toric varieties}
\author{Mandira Mondal, V. Trivedi}
\address{School of Mathematics, Tata Institute of Fundamental Research,
Homi Bhabha Road, Mumbai-400005, India}
\email{mandira@math.tifr.res.in; vija@math.tifr.res.in}
\date{}

\begin{abstract} For  a toric pair $(X, D)$, where $X$ is a projective toric variety 
of dimension $d-1\geq 1$ and $D$ is a very ample $T$-Cartier divisor, 
we show that the Hilbert-Kunz density function
$HKd(X, D)(\lambda)$ is the $d-1$ dimensional volume of 
${\overline \sP}_D \cap \{z= \lambda\}$, where 
${\overline \sP}_D\subset \R^d$ is a compact $d$-dimensional set (which is a
finite union of convex polytopes). 

We also show that, for $k\geq 1$, the function
 $HKd(X, kD)$ can be replaced by another compactly 
supported continuous function $\varphi_{kD}$ which 
is `linear in $k$'. 
This gives 
the formula for the associated  coordinate ring $(R, {\bf m})$: 
$$\lim_{k\to \infty}\frac{e_{HK}(R, {\bf m}^k) - e_0(R, {\bf m}^k)/d!}{k^{d-1}}
= \frac{e_0(R, {\bf m})}{(d-1)!}\int_0^\infty\varphi_D(\lambda)d\lambda, $$
where $\varphi_D$ (see Proposition~\ref{p5}) is solely
detemined by the shape of the polytope $P_D$, associated to the toric pair 
$(X, D)$. 
Moreover $\varphi_D$ is a multiplicative 
function for Segre products.

This yields explicit computation of  $\varphi_D$ (and hence the limit), 
for smooth Fano toric surfaces
with respect to anticanonical divisor.
In general, due to this formulation in terms of the polytope $P_D$, one 
can  explicitly 
compute the limit for two dimensional toric pairs and their 
Segre products.

We further show  that (Theorem~\ref{t4}) the renormailzed limit takes the 
minimum value if and only   if
the polytope $P_D$ {\it tiles the space} $M_{\R} = \R^{d-1}$ 
(with the lattice $M = \Z^{d-1}$). As a consequence, one gets an algebraic 
formulation of  
the tiling property of any rational convex polytope.

\end{abstract}
\maketitle 

\section{Introduction}
Let $R$ be a  Noetherian ring of prime characteristic
$p >0$ and of dimension $d$ and let
$I\subseteq R$ be an ideal of finite colength. Then we recall that
the Hilbert-Kunz
multiplicity of $R$ with
respect to $I$ is defined as
$$e_{HK}(R, I) = \lim_{n\to \infty}\frac{\ell(R/I^{[q]})}{q^{d}},$$
where
$q=p^n$, $ I^{[q]} = n$-th Frobenius power of $I$
= the ideal generated by $q$-th powers of elements of $I$. This
 is an ideal of finite colength and  $\ell(R/I^{[q]})$ denotes the
length of
the
$R$-module $R/I^{[q]}$. Existence of the limit was proved by
Monsky [Mo1]. This invariant has been extensively
studied, over the years (see the survey article [Hu]). As various 
standard techniques, used for studying multiplicities, are not applicable 
for the 
invariant $e_{HK}$, it has been difficult to compute (there is no general formula even for 
a hypersurface).

 In order to study $e_{HK}$,  
when  $R$ is a standard graded ring ($\dim~R\geq 2)$ 
and $I$ is a homogeneous ideal of finite colength, 
the second author (in [T2]) has defined the notion of Hilbert-Kunz Density function
and its relation with the HK-multiplicity (stated in this paper as Theorem~\ref{hkd}): 
{\it the HK density function is a compactly supported continuous function 
$HKd({R,I}):[0,\infty) \longto 
[0, \infty)$ such that 
$$e_{HK}(R, I)=\int_0^\infty HKd({R, I})(x) \ dx.$$}

Further using this relation, the asymptotic behaviour of 
$e_{HK}(R, I^k)$ as $k\to \infty$, was studied in [T3].

The asymptotic behaviour of $e_{HK}$ was first studied by
Watanabe-Yoshida in [WY1],
 for a Noetherian local ring $(R, {\bf m})$ of dimension $d\geq 2$ and  an
${\bf m}$-primary ideal $I$. In particular, in [WY1] it is shown
that
 $$\frac{e_0(R, I^k)}{d!} \leq  e_{HK}(R, I^k) \leq
\frac{{{k+d-1}\choose{d}}}{k^d}e_0(R, I^k),$$
and as a corollary they get
$$e_{HK}(R, I^k) = \frac{e_0(R, I)}{d!}k^d + o(k^d).$$
Later Hanes in [Ha] (Theorem~3.2) improved this as follows:
$$\ell(R/I^{[q]k}) =
\left[\frac{e_0(R, I)}{d!}k^d + O(k^{d-1})\right]q^d.$$

In other words 
$${e_{HK}(R, I^k) -e_0(R,I^k)/d!} =
O(k^{d-1}).$$

In [T2] (Theorem~3.6), the second author proved the following result:

\vspace{5pt} 
 
\noindent{\bf Theorem}\quad {\it Let $R$ be a standard graded ring of dimension~$d\geq 2$
over a perfect
field $K$ of characteristic $p >0$, and
let $I\subset R$ be a homogeneous ideal of finite colength, which has a set of
generators of the same degree.
Let $M$ be a finitely generated graded $R$-module. Then
$$\lim_{k\to \infty} \frac{e_{HK}(M, I^k) -
e_0(M,I^k)/d!}{k^{d-1}} = \frac{e_0(M,I)}{2(d-2)!} - \frac{E_1(M, I)}{(d-1)!},$$
where 
$$E_1(M,I) := \lim_{q\to \infty}e_1(M, I^{[q]})/q^d~~~\mbox{exists}. $$}

\vspace{5pt}
In particular, it implies
$$\ell(M/I^{[q]k}M) = \left[\frac{e_0(M, I)}{d!}k^d +
\left(\frac{e_0(M,I)}{2(d-2)!} -
\frac{E_1(M, I)}{(d-1)!}\right)k^{d-1} + o(k^{d-1})\right]q^d + O((kq)^{d-1}).$$

The above limit can be computed in the case of a nodal plane curve (due to [Mo2]), 
and in the case of elliptic curves
and  full flag varieties  (due to [NT]).
Other known cases are Hirzebruch surfaces ([T1]).

In this paper we study the same question for a projective toric variety 
$X$ of dimension $d-1\geq 1$ over an algebraically closed field $K$ of 
characteristic $p>0$, with a very ample $T$-Cartier divisor $D$. 
Here,  by $HKd(X, D)$ (or  $e_{HK}(X, D)$) for a pair $(X, D)$ 
we mean 
the HK density function (or HK multiplicity, respectively) of the associated 
homogeneous  coordinate ring with respect to its graded maximal ideal.

It is well known that such a pair 
$(X, D)$ corresponds to a lattice polytope (that is, the convex hull of 
a finite set of lattice points)
$P_D \subseteq M_{\R} \simeq \R^{d-1}$ 
(see (\ref{*}) for the definition).

For a pair $(K[H], I)$, where  $K[H]$ is a toric ring (= normal semigroup ring) and 
$I$ is a monomial ideal $I$ (such that $\ell(K[H]/I) <\infty$), 
K.~Watanabe (Theorem~2.1 of [W]) has proved that the  
$e_{HK}(K[H],I)$ is a rational number.

Later K.~Eto (in [E]) proved the following result:

\vspace{5pt}

\noindent{\bf Theorem}\quad (Theorem~2.2, [E])\ : 
{\it Let $S$ be an affine semigroup and $a_1, \ldots , a_v \in S(\subset \mathbb{Z}^N )$ elements 
such that $K[S]/J$ has finite length, where $J = (x^{a_1} ,\ldots, x^{a_v})$. Let 
$C$ denote the convex rational polyhedral cone spanned by $S$ in $\mathbb{R}^N$ and 
$\mathcal{P}=\{p\in C\ |\ p\notin a_j +C\ \text{for each}\ j\}$. Then
$$e_{HK}(K[S], J) = Vol(\overline{\mathcal{P}}) ,$$
where $\overline{\mathcal{P}}$ is the closure of $\mathcal{P}$ and $Vol$ denotes the 
relative volume ([St2], p. 569)}.

\vspace{5pt}

For a toric pair $(X, D)$ as above (see Section~2 for the detailed theory), 
if $C_D$ is  the convex rational polyhedral 
cone spanned by $P_D\times 1$ in 
$M_{\mathbb{R}}\times\mathbb{R}$ and if
$${\sP_D}=\{p\in C_D\ |\ p\notin 
(u_j,1) +C_D\ \mbox{for every}~u_j\in P_D\cap \Z^{d-1}\}$$  
then, by the above theorem of K. Eto, we have   
$$e_{HK}(X, D) = Vol(\overline{\sP}_D)~~~\mbox{and}~~~e_{HK}(X, kD) = 
Vol(\overline{\sP}_{kD}).$$

As in [T3], we will study the asymptotic behaviour of $e_{HK}(X, kD)$ (as $k$ varies),
  via 
HK-density functions. However we do not use the results of [T3]: instead we 
directly interpret the
HK density function (as in [T2]) for a toric pair $(X, D)$, in terms of $\sP_{D}$:
 
\begin{thm}\label{t2}
Let $C_D$ denote the convex rational polyhedral cone spanned by $P_D\times 1$ in 
$M_{\mathbb{R}}\times\mathbb{R}$.  Let 
$$\mathcal{P_D}=\{p\in C_D\ |\ p\notin 
(u_j,1) +C_D\ \mbox{for every}~u_j\in P_D\cap \Z^{d-1}\}.$$  Then the Hilbert-Kunz density 
function $HKd(X, D)$ is given by the sectional volume of 
$\overline{\mathcal{P}}_D$ ($\overline{\mathcal{P}}_D$ is the closure of $\sP_D$), 
{\it i.e.}
precisely,  
$$HKd(X, D)(\lambda) = Vol_{d-1}(\overline{\mathcal{P}}_D\cap \{z=\lambda\}),$$
for $\lambda\geq 0 $ (note that the relative volume and the volume are same here). 
\end{thm}

We prove the following key proposition:

\begin{propose}\label{p5}
Let $(X, D)$ be a toric pair. Then, for $\lambda \geq 0$,
$$HKd(X, kD)(\lambda +1)= 
\frac{e_0(X, D)k^{d-1}}{(d-1)!}\varphi_{kD}(\lambda)+O(k^{d-2}),$$ 
where $\varphi_{kD}:[0, \infty) \longto [0, 1]$  is the 
compactly supported continuous function
given by 
$$\varphi_{kD}(\lambda) =  \mbox{Vol}_{d-1} ([
W_v\times\{z=\lambda\}]\setminus \bigcup_{u\in\Z^{d-1}} \left[(u,1)+
C_{kD}\right]),$$
for any vertex $v\in \Z^{d-1}$, where $W_v\subset \R^{d-1}$ (as in 
Notations~\ref{n1})  is 
the $d-1$ dimensional unit cell at the vertex $v$.
\end{propose}

  Note that $(u,1)+C_{kD}$ is the translate of the cone 
$C_{kD}\subseteq \R^d$ from its vertex  $\underline{0}$ to the vertex at $(u,1)$.
In fact (see Remark~\ref{***})
\begin{equation}\label{*2}\varphi_{kD}(\lambda) = \varphi_{D}(k\lambda),
~~~\mbox{for all}~~~~ 
\lambda \geq 0.\end{equation} 
Hence 
\begin{equation}\label{**}
\int_0^{\infty}\varphi_{kD}(\lambda)d\lambda = \frac{1}{k}\int_0^{\infty}
\varphi_{D}(\lambda)d\lambda.\end{equation}
In otherwords, we have `replaced'  the continuous function  $HKd(X, kD)$ by another  
continuous function 
$\varphi(kD)$ which is `linear in $k$', in the sense of (\ref{*2}) and (\ref{**}).

Now the equality given in the above Proposition 
combined with Theorem~1.1 of [T2] gives the main result of this paper:

\begin{thm}\label{t1}For  a projective toric variety $X$ with a very ample $T$-Cartier 
divisor $D$, we have 
$$\lim_{k\to \infty} \frac{e_{HK}(R, {\bf m}^k) - e_0(R,{\bf m}^k)/d!}{k^{d-1}} = 
\frac{e_0(R,{\bf m})}{(d-1)!}\int_0^{\infty}\varphi_{D}(\lambda)d\lambda, $$
where $\varphi_{D}:[0, \infty)\longto [0, \infty)$, defined 
as before (for $k=1$), is a compactly 
supported continuous function and is solely determined by the shape 
of the polytope $P_D$ (as in (\ref{*})) associated to the toric pair $(X, D)$.
\end{thm}

 In fact,
 Theorem~\ref{t4} states that 
 among the set of $d-1$ dimensional toric pairs $(X, D)$,
(the renormalized) limit of Theorem~\ref{t1} achieves the  minimum if and only 
 if the polytope $P_D$ tiles the space $M_{\R} = \R^{d-1}$ with lattice $M$.
 In other words 
the asymptotic behaviour of $e_{HK}(R, {\bf m}^k)$ (relative to 
its usual multiplicity $e_0(R, {\bf m})$) as $k\to \infty$, 
characterizes the 
tiling property of the associated polytope $P_D$ (with the canonical lattice 
$\Z^{d-1})$.
Similarly, the  tiling property of any rational convex
polytope can be formulated (Remark~\ref{pt}) in terms of 
this algebraic invariant (the renormalized limit).

It is easy to see that the polytope associated to   
the $d-1$ dimensional  Segre self-product of the toric pair 
$(\P^1, \sO(m_0))$, for any $m_0\geq 1$ tiles the space $\R^{d-1}$.
Hence the renormalized limit for such toric pair achieves the minimum in 
any dimension.  
In particular  this result (Remark~\ref{r4}) is also in the spirit of the well 
known conjecture of Watanabe-Yoshida (Conjecture~4.2, [WY2]).

Moreover, similar to the  HKd functions, the function   
$\varphi_{D}$ turns out to have a multiplicative property on the set of toric pairs:
\begin{propose}\label{t3}Let $(X, D)$ and $(Y,D')$ be two  toric pairs defined 
over the same perfect field $k$.
Then  
$$(1-\varphi_{X\times Y, D\boxtimes D'})=(1-\varphi_{X, D})(1-\varphi_{Y, D'}),$$
where  $(X\times Y, D\boxtimes D')$ is the toric variety given by the Segre product 
of the toric verieties $(X, D)$ and $(Y, D')$ and $D\boxtimes D'$ denotes a 
divisor corresponding to the line bundle $\pi_1^*\mathcal{O}_X(D)\otimes
\pi_2^*\mathcal{O}_Y(D')$, where $\pi_1:X\times Y\to X, \pi_2 :X\times Y\to Y$ are 
the two projection morphisms.
\end{propose}

We also compute the function $\varphi_D$ for all 
five smooth Fano toric surfaces with respect to their respective 
anticanonical divisors.
Similarly one can explicitly compute $\varphi_D$, for every 
two-dimensional toric pair $(X, D)$.
Hence, due to the multiplicative property (Proposition~\ref{t3}), 
one can compute 
$\varphi_{X,D}$, where $(X, D)$ is a Segre product of the two dimensional toric pairs.
In particular, one can compute the limit (Theorem~\ref{t1}), in these cases.

\vspace{5pt}

The organization of the paper is as follows.

\vspace{5pt}

In Section~2 we recall some notations about toric varieties  
(following Fulton [Fu]), in a form useful for us.

\vspace{5pt}
In Section~3 we give a self contained  proof of 
the fact that the sectional volume function $\phi(\lambda) = \mbox{Vol}_{d-1}
{\tilde P} \cap \{z= \lambda\}$, where ${\tilde P}$ is a $d$-dimensional convex polytope
in $\R^d$ with no facets lying in hyperplanes parallel to $\{z=0\}$, is a continuous piecewise polynomial function of $\lambda$.

\vspace{5pt}

In Section~4 we give a proof of Theorem~\ref{t2}, relating the HKd function for $(X, D)$
with the sectional volume of $\sP_D$. We prove that $\sP_D$ can be written as a 
finite union of convex polytopes with disjoint interiors, and none of  the facets 
of the involved polytopes lie in the $\{z =\lambda \}$ hyperplane for any $\lambda$.  
Now, owing to the fact that the  
HKd function and the sectional volume function (as given in Section~3) 
are both continuous, we only need to check the equality for a suitable 
dense set, namely the set of rationals  
$\{m/p^n\mid m, n\in \Z_{\geq 0}\} \subseteq \R_{\geq 0}$.

\vspace{5pt}

Section~5 involves purely convex geometry.
In this section  we prove that, for any integer $k\geq 1$,  
$\sP_{kD}\cap \{z=1+\lambda\} = \emptyset $, 
for $\lambda \geq l/k$ (where $l$ is the number of the vertices of the polytope $P_D$).
This implies $\sP_{kD} \subseteq P_{(k+l)D}\times \R_{\geq 0}$, {\it i.e.}, $\sP_{kD}$ 
lies `approximately' in a cylinder over the polytope $P_{kD}$.   
We also prove various properties of the function $\varphi_{kD}$ here.
Lemma~\ref{l4} implies that, in the definition of $\varphi_{kD}$ 
(defined with respect to a fixed unit cell $W_v$ given by  a  vertex $v\in \Z^{d-1}$, as in 
Proposition~\ref{p5}),
we can replace the infinite set $\{u\in \Z^{d-1}\}$ by a finite set of a fixed size, 
{\it i.e.}, by 
$$\{u\in \Z^{d-1}\cap B(v, r)\mid W_v\subseteq B(v, r)\},$$ 
where $r$ is independent of $k$.
This implies  that 
$$ \cup_{v\in S} \left((W_v\times \R_{\geq 1})\cap \sP_{kD}\right)  \subseteq \sP_{kD} 
\cap \{z\in R_{\geq 1}\} \subseteq 
\cup_{v\in S'}\left((W_v\times \R_{\geq 1})\cap \sP_{kD}\right),$$
where the set $S$ and $S'\subseteq \Z^{d-1}$ have sufficiently large overlap 
(note  that 
$W_v\times \R_{\geq 1}$ is the cylinder over the unit cell  $W_v$), and for a 
 `general' $v$  from either set,
$$\mbox{Vol}_{d-1}\left((W_{v}\times \{z= 1+\lambda\})\cap\sP_{kD}\right) = 
\varphi_{kD}(\lambda),~~\mbox{for all}~~\lambda\geq 0,$$ 
where we already know that $\varphi_{kD}$ is `linear' with respect to $k$ (see (\ref{*2}) and 
(\ref{**})).

In Section~6, we use the above results to prove the key Proposition~\ref{p5},
 which replaces $HKd(X, kD)$ by 
$\varphi_{kD}$ upto $O(k^{d-2})$. 
Next in this section we prove the main Theorem \ref{t1},  and 
  Proposition~\ref{t3} gives the multiplicative 
property of the function $\varphi_D$. Theorem~{\ref{t4} and Remark~\ref{pt}  
relate the tiling of $P_D$ with lattice $M$ and 
the asymptotic growth of the HK multiplicity for $(X,D)$.    

Section~7 consists of examples.  We prove Theorem~\ref{t1}, 
for a toric pair $(\P^1, D)$, which takes care of one dimensional toric pairs.
We also compute $\varphi_D$ (and hence the limit (Theorem~\ref{t1}), 
for the smooth Fano toric surfaces with respect to their 
anticanonical divisors. 


\section{preliminaries}
Henceforth we assume that  $K$ is an algebraically closed field of 
char. $p > 0$. We follow the notations from [Fu].
Let  $N$ be  a lattice (which is isomorphic to 
$\Z^n$)
and let  $M = Hom(N, \Z)$ denote the dual lattice 
with a dual pairing $\langle\ , \rangle$.  
Let $T = \text{Spec} (K[M])$ be the torus with character lattice $M$. Let $(X, D)$ denote a complete toric variety over $K$ with fan $\Delta \subset N_\R$ and very ample $T$-divisor $D$ on $X$.

We recall that the $T$-divisors on $X$ (the 
irreducible subvarieties of 
codimension $1$ which  are $T$-stable) correspond to one dimensional cones 
(which are edges/rays of $\Delta $) of $X$. If  $\tau_1,\ldots, \tau_n$ denote the edges of 
 the fan $\Delta$, then these divisors are the orbit closures $D_i= V(\tau_i)$.
 A $T$-divisor $D=\sum_i a_iD_i$ (note that $a_i$ are integers)  determines a  
lattice polytopes in $M_{\mathbb{R}}$ defined by
\begin{equation}\label{*}P_D=\{u\in M_{\mathbb{R}} \ | \ \langle u, v_i\rangle\geq -a_i ~~\text{for all}\ 
i\ \}\end{equation}
and the induced embedding of $X$ in $\P^{r-1}$ is given by 
$$\phi=\phi_D: X\to\P^{r-1},\ 
\ x\mapsto ({\chi}^{u_1}(x):\ldots: {\chi}^{u_r}(x)),$$ 
where $P_D\cap M=\{u_1, u_2,\ldots, u_r\}$. 

Moreover the 
global sections of the line bundle $\mathcal{O}(D)$ are
$$\Gamma(X, \mathcal{O}(D))=\bigoplus_{u\in P_D\cap M}K. \chi^u.$$ 
For  any integer 
$m\geq 1$, we have $P_{mD}=mP_D$ (see Page~67 in [Fu]).

For $(X, D)$ and $P_D$  as above.  Consider $\sigma$ the cone in 
$N\times \mathbb{Z}$ whose dual ${\sigma}^{\vee}$ is the cone over $P_D\times 1$ 
in $M\times \mathbb{Z}$. Then the affine 
variety $U_{\sigma}$ corresponding to the cone $\sigma$ 
is the affine cone of $X$ in  $\mathbb{A}^r_K$.

If $S$ is the semigroup generated by $\{(u_1, 1), \ldots, (u_r,1)\}$ then 
the homogeneous coordinate ring of $X$ (with respect to this 
embedding) is 
$K[S] = K[{\chi}^{(u_1,1)},
\ldots, {\chi}^{(u_r,1)}]$. Note that
there is an isomorphism of graded rings (see Proposition 1.1.9, [CLS])
$$ \frac{K[Y_1,\ldots, Y_r]}{I}\simeq K[{\chi}^{(u_1,1)},
\ldots, {\chi}^{(u_r,1)}] = K[S],$$
 where, the kernel
$I$ is  generated by the binomials  of the form
$$Y_1^{a_1}Y_2^{a_2}\cdots Y_r^{a_r}-Y_1^{b_1}Y_2^{b_2}\cdots Y_r^{b_r}$$
where $a_1,\ldots, a_r, b_1,\ldots, b_r$ are nonnegative integers satisfying 
the equations 
$$a_1u_1+\cdots+a_ru_r=b_1u_1+\cdots+b_ru_r\ \ \text{and}\ \ a_1+\cdots+a_r=
b_1+\cdots+b_r.$$

\begin{defn}\label{dtp}By a toric pair $(X, D)$, we mean $X$ is a projective variety of
dimension $d-1\geq 1$ over a field $K$  with a very ample 
$T$-divisor $D$. Moreover $P_D$ denotes the associated lattice convex polytope as defined by 
 (\ref{*}). 
The homogeneous coordinate ring of $X$ with respect to this embedding is 
\begin{equation}\label{*1} K[S]=K[{\chi}^{(u_1,1)},\ldots, {\chi}^{(u_r,1)}],\end{equation} 
where 
$P_D\cap M=\{u_1,
\ldots, u_r\}$ and  $S$ is the semigroup generated by $\{(u_1,1),\ldots,(u_r,1)\}$.
\end{defn}

Note that due to this isomorphism, we can consider $K[S]$ as a standard graded  
ring, where 
$\deg~\chi^{(u_i, 1)} = 1$.
While dealing with the cone in $M_\mathbb{R}\times\mathbb{R}\simeq 
\mathbb{R}^{d}$, 
we denote the last co-ordinate as $x_d$ or $z$, interchangebly. 

\begin{rmk}\label{St}We recall the following well known fact (see [St2], Excercise~33 
and [St1], Proposition~4.6.30).

If $P$ is a $d$-dimensional rational convex polytope in $\R^m$ and $i(P, n) = 
\#(nP\cap\Z^m)$ then 
$$i(P, n) = c_d(n)n^d+c_{d-1}n^{d-1}+\cdots+c_0(n),$$
where $c_0, \ldots, c_d$ are periodic functions of $n$ and 
$c_d(n) = \mbox{Vol}_d(P)$.
\end{rmk} 

\section{Volume of \enquote{slices} of convex polytope}
Let $P$ be a $d$-dimensional convex polytope in $\mathbb{R}^d$.
For $Q\subseteq \R^d$ we denote $Q\cap \{z=\lambda\} = 
Q\cap \{(\underline{x}, \lambda)\mid \underline{x}\in \R^{d-1} \}
\subseteq \R^d$

Our goal in this section, is to describe the 
behaviour of the function $\phi:(-\infty, \infty) \longto [0, \infty)$ given by 
$\phi(t) = \text{Vol}_{d-1}(P \cap  \{z=t\})$.

\begin{defn}\label{dd1}
Let $\pi:\mathbb{R}^{d}\longrightarrow \mathbb{R}$ be the projection map 
given by projecting to the last co-ordinate $z$. 
Then we denote the set $\pi(\mbox{vertex set of } P) =  
\{\tau_1, \ldots, \tau_m\}$, where $\tau_1 < \tau_2 <
\cdots <\tau_m$.  
\end{defn}

\begin{lemma}
\label{lemoneside}\begin{enumerate}
\item  The support of $\phi(t)$ is a compact connected interval.
\item  Suppose $P$ has no supporting hyperplane parallel to 
the hyperplane $\{z=0\}$. Then $\phi(\tau_0)=\phi(\tau_m)=0$.\end{enumerate}
\end{lemma}
\begin{proof}
First we prove that, if  $\{z=\alpha\}$ is a hyperplane in $\mathbb{R}^d$ such that 
dimension $\{z=\alpha\}\cap P \leq d-2$. Then 
it can not pass through the interior of the polytope and 
therefore $P$ lies entirely in one of the closed half spaces defined by $\{z=\alpha\}$.

Suppose by contradiction, 
$x\in \{z=\alpha\}\cap \text{int}(P)$. Let $B^d(x, \epsilon)$ be a small 
ball around $x$ of radius $\epsilon$ inside $P$. Then 
$B^d(x, \epsilon)\cap\{z=\alpha\}\cap P$ 
is a nonempty  $d-1$ dimensional ball. Hence dimension $\{z=\alpha\}\cap P$ 
is $d-1$, which is a contradiction.

Suppose the support of $\phi$ is not connected then we have $a< x<b$ in $\R$ such that
$\phi(a)\neq 0, \phi(b)\neq 0$ and $\phi(x)=0$. But 
then dimension $\{z=x\}\cap P \leq d-2$, therefore $P$ lies in one side of the 
hyperplane $\{z=x\}$, which is a contradiction since both $\phi(a)$ and $\phi(b)$ is 
nonzero. Further, since $P$ is a bounded polytope, support of $\phi$ is a compact 
interval $\subseteq [\tau_0, \tau_m]$.

 Suppose $\phi(\tau_m)\neq 0$ then $\dim~P\cap \{z = \tau_m\} = d-1$. 
Since for any $\epsilon >0$, $P\cap \{z= \tau_{m+\epsilon}\} = \emptyset $, the  hyperplane
$\{z=\tau_m\}$ does not pass through the interior of $P$. Hence $\{z=\tau_m\}$ is 
a supporting hyperplane of $P$ parallel to $\{z= 0\}$, which is a contradiction.
Similar proof shows $\phi(\tau_0)$ is $0$.
\end{proof}

A volume formula $\phi$ for "slices" of a simplex has been derived by  
C.A. Micchelli ([Mi], Chapter 4) in more general context, using 
the univariate B-splines. For details about B-splines and volume of slices, 
see [CS]. Here we give a simpler self contained proof, which is  suited  to 
our case.

\begin{lemma}\label{thPP}Let $S_{ik}\subset\R^d$ be a $d$-simplex such that the set of 
vertices of $S_{ik}$ are contained in $\{z=\tau_i\}\cup \{z=\tau_{i+1}\}$. 
Then the function $\phi_{ik}:[\tau_i,\tau_{i+1}] \longrightarrow \mathbb{R}$, 
given by $\lambda \mapsto \mbox{Vol}(S_{ik}\cap \{z=\lambda\})$ is a 
polynomial function of degree $\leq d-1$ in $\lambda$.
\end{lemma}
\begin{proof}By the  hypothesis  $S_{ik}\cap \{z=\tau_i\} \simeq \Delta_r $ is $r$-simplex
given by the  
vertices $v_0, \ldots, v_r$  and $S_{ik}\cap \{z=\tau_{i+1}\}\simeq 
\Delta_s $ is $s$-simplex given by the vertices $w_0, \ldots, w_s$, 
where $\{v_0, \ldots, v_r, w_0, \ldots, w_s\}$ the vertex set of $S_{ik}$.  Note that 
since $r+1+s+1 = d+1$, we have $r+s = d-1$.

Let $\lambda \in [\tau_i, \tau_{i+1}]$. Let 
$\lambda_1 = \frac{\tau_{i+1}-\lambda}{\tau_{i+1}-\tau_i}$ and    
$\lambda_2 = \frac{\lambda-\tau_i}{\tau_{i+1}-\tau_i}$.
Then 

\noindent{\bf Claim}~(1) $S_{ik}\cap\{z=\lambda\} = 
\{\lambda_1(p_0+v_0)+\lambda_2
(p_1+w_0)\mid p_0\in \Delta_r-v_0,~~p_1\in \Delta_s-w_0\}$.

\noindent{Proof of the claim}: Any element $p$ of $S_{ik}\cap \{z=\lambda\}$ can  be written as  $p = \sum_{i=0}^r a_iv_i+\sum_{j=0}^{s}b_jw_j$, where $a_i, b_j \geq 0$ and 
$\sum_{i=0}^{r}a_i+\sum_{j=0}^{s}b_j = 1$.
Therefore 
$$p=\lambda_1\left(\frac{\sum_{i=1}^{r}a_i(v_i-v_0)}{\lambda_1}\right)+ \lambda_1 v_0+\lambda_2\left(\frac{\sum_{j=1}^{s}b_j(w_j-w_0)}{\lambda_2}\right)+\lambda_2w_0.$$
This proves the claim.

\noindent{\bf Claim}~(2)\quad $\{v_1-v_0, \ldots, v_r-v_0, w_1-w_0, \ldots, w_s-w_0\}$
is a basis of $\mathbb{R}^{d-1}$.

\noindent{Proof of the claim}: Note that for a choice of $\lambda \in (\tau_i,\tau_{i+1})$, the convex polytope $S_{ik}\cap\{z= \lambda\}$ is $d-1$ dimensional 
 (as the hyperplane $\{z= \lambda\}$ contains no vertices of $S_{ik}$, but the hyperplanes  $\{z=\tau_i\}$ and $\{z=\tau_{i+1}\}$ both contain some vertices 
of $S_{ik}$, we deduce that the hyperplane $\{z=\lambda\}$ intersects the interior of $S_{ik}$).
By Claim~(1), the set of $d-1$ vectors $\{v_1-v_0, \ldots, v_r-v_0, w_1-w_0, \ldots, w_s-w_0\}$ generate the $d-1$ dimensional convex set $S_{ik}\cap \{z= \lambda\}-(\lambda_1v_0+\lambda_2w_0)$. This proves the claim.

Let ${\widetilde{\Delta_{rs}}}$ denote the image of the map $\psi_{v,w}:\Delta_r\times \Delta_s \longrightarrow \mathbb{R}^{d-1}$ given by $(p_0, p_1)\mapsto p_0+p_1$. Since $\Delta_r$ and $\Delta_s$ are convex polytopes, the set  ${\widetilde{\Delta_{rs}}}$ is a convex polytope and of dimension $d-1$.
Now for a given  $\lambda \in [\tau_i,\tau_{i+1}]$, we can define the linear transformation 
$T_{\lambda}:\R^{d-1}\longrightarrow \R^{d-1}$ given by $\sum_i\alpha_i(v_i-v_0)+\sum_j\beta_j(w_j-w_0)\mapsto \lambda_1\sum_i\alpha_i(v_i-v_0)+\lambda_2\sum_j\beta_j(w_j-w_0)$ 
(this is a well defined map due to Claim~(2)).

Note that, for any $\lambda\in [\tau_i,\tau_{i+1}]$, $(S_{ik}\cap \{z=\lambda\})-
(\lambda_1v_0+\lambda_2w_0) =
T_\lambda({\widetilde{\Delta_{rs}}})$. Therefore 
$$\mbox{Vol}(S_{ik}\cap \{z=\lambda\}) = \frac{\mbox{Det}~
(T_\lambda)}{(d-1)!}\mbox{Vol}({\widetilde{\Delta_{rs}}}) = 
\frac{(\tau_{i+1}-\lambda)^r(\lambda-\tau_i)^s}{(\tau_{i+1}-\tau_i)^{r+s}(d-1)!} 
\mbox{Vol}(\widetilde{\Delta_{rs}}).$$
This proves the lemma.
\end{proof}

\begin{thm}
\label{thvol}
Let $P$ be a bounded full dimensional convex polytope in $\mathbb{R}^d$ 
which has no supporting hyperplane parallel to the hyperplane $\{z=0\}$. 
Then 
\begin{enumerate}
\item the function $\phi(t)=Vol_{d-1}(P\cap\{z=t\})$ is a polynomial of 
degree $\leq d-1$ 
on $(\tau_i, \tau_{i+1})$, for $i=0,\ldots, m-1$. Moreover 
\item $\phi$ is continuous on all of $\mathbb{R}$.\end{enumerate}
\end{thm}
\begin{proof}
(1)\quad Let $$P_{[\tau_i,\tau_{i+1}]} = \{p\in P\mid \tau_i\leq \pi(p) \leq \tau_{i+1}\}.$$
Note that $P_{[\tau_i,\tau_{i+1}]}$ is a convex polytope with vertices only at the 
level $\{z=\tau_i\}$ and $\{z= \tau_{i+1}\}$.  
For $t\in [\tau_i, \tau_{i+1}]$, we have 
$\phi(t) = \phi\mid_{P_{[\tau_i,\tau_{i+1}]}}(t)$. Therefore, it is enough to show that 
$\phi_i:= \phi\mid_{P_{[\tau_i,\tau_{i+1}]}}:[\tau_i,\tau_{i+1}] 
\longrightarrow [0,\infty)$
is a polynomial function in $\lambda$ of degree $\leq d-1$, for $i=0, \ldots, m-1$.

We take a triangulation (see [L]) of $P_{[\tau_i, \tau_{i+1}]}$ such 
that vertices of each triangulating simplex are vertices of $P_{[\tau_i, \tau_{i+1}]}$ 
itself.

Hence we can triangulate $P_{[\tau_i,\tau_{i+1}]} = \cup_{k=1}^{L_i}S_{ik}$ 
in $d$-simplices such that the  vertex set  of each simplex $S_{ik}$ is a subset  
of the vertex set of $P_{[\tau_i,\tau_{i+1}]}$.
Since vertices of $S_{ik}$ lie in $\{z=\tau_i\}$ and $\{z=\tau_{i+1}\}$, 
if $t\in(\tau_i, \tau_{i+1})$, where
$i=0,\ldots, m-1$, the plane $\{z=t\}$ does  not contain any face of $S_{ik}$.
  Therefore dimension of $S_{ik}\cap S_{ik'}\cap\{z=t\}$ is 
$< d-1$. For $t=\tau_i$, if dimension of  $S_{ik}\cap S_{ik'}\cap\{z=\tau_i\}$ is 
$d-1$, then dimension of both $S_{ik}\cap\{z=\tau_i\}$ and $S_{ik'}\cap\{z=\tau_i\}$ 
is $d-1$, it follows that $S_{ik}\cap\{z=\tau_i\}=S_{ik'}\cap\{z=\tau_i\}$. 
Hence for $x\in S_{ik}\cap\{z=\tau_i\}$, one can find $\epsilon>0$, such 
that $B^{d}(x, \epsilon)\cap \{z\geq \tau_i\}\subset S_{ik}\cap S_{ik'}$, 
a contradiction. We show that, for $t\in [\tau_i, \tau_{i+1}]$
$$\phi_i(t)=\sum_{k=1}^{L_i}\phi_{ik}(t) ,$$
where $\phi_{ik}(t)=\text{Vol}_{d-1}(S_{ik}\cap\{z=t\})$ is the volume 
function for the simplex $S_{ik}$, $k=1,\ldots, L_i$. Enough to show 
$$\text{Vol}_{d-1}\big((\cup_{k=1}^{l}S_{ik})  \cap \{z=t\}\big)=\sum_{k=1}^{l}
\phi_{ik}(t)$$
for $1\leq l\leq L_i$. This easily follows by induction because dimension of 
$S_{ik}\cap S_{ik'}\cap\{z=t\}$ is $< d-1$. 
This proves part on of the theorem.

 For the second part, it is 
enough to show that $\phi$ is continuous at $\tau_0$ and $\tau_m$. 
Since $\phi(\tau_0)=\phi_0(\tau_0)=0$, so is $\phi_{0k}(\tau_0)$, 
for $k=1,\ldots, L_0$. By Lemma \ref{thPP} each $\phi_{0k}$ is continuous at 
$\tau_0$. Hence so is $\phi$. similarly, $\phi$ is continuous at $\tau_m$.  
\end{proof}

\section{Hilbert-Kunz-Density function}
In [T2], the second author has defined the notion of Hilbert-Kunz 
Density function, and given its relation with the HK-multiplicity. We use the 
following interpretation of the HK multiplicity via the HK 
density function.

\begin{thm}~({\mbox{Theorem 1.1 in [T2]}})
\label{hkd}
Let $R$ be a standard graded Noetherian ring of dimension $d\geq 2$ over an algebraically 
closed field $K$ of char $p > 0$, and let $I\subset R$ be a homogeneous ideal  such that 
$l(R/I) < \infty$.
For $n \in \mathbb{N}$ and 
$q = p^n$, let $f_n:[0,\infty)\longto [0,\infty)$
be defined as  
$$f_n(R, I)(x) = \frac{1}{q^{d-1}}\ell(R/I^{[q]})_{\lfloor xq\rfloor}.$$

Then
 $\{f_n(R, I)\}$ converges uniformly to a compactly supported continuous function
 $f_{R, I}:[0, \infty)\longto [0, \infty)$, where  $f_{R,I}(x)=
\text{lim}_{n\to\infty} f_n(R,I)(x).$
and
$$e_{HK}(R, I)=\int_0^\infty f_{R, I}(x) \ dx.$$
\end{thm}

\begin{defn}\label{d2}For a given pair $(X, D)$ (Definition~\ref{dtp})
 we  have the associated standard graded ring $K[S]$.
We define the associated density function $HKd(X, D) = HKd(K[S], {\bf m})$,
 where ${\bf m}$ is the graded maximal ideal of  $K[S]$.
Therefore, for $q=p^n$ where $n\geq 1$, 
$$HKd(X, D)(\lambda)= \lim_{n\to \infty} f_n(\lambda)
= \lim_{n\to \infty}\frac{1}{q^{d-1}}
\ell\left(\frac{K[S]}{{\bf m}^{[q]}}\right)_{\lfloor q\lambda \rfloor}.$$ \end{defn}

\begin{notations}\label{n2}In $\R^d$, we denote the last ($d^{th}$) 
coordinate by 
$z$. Let  $\lambda \in \R_{\geq 0}$. Then
\begin{enumerate}
\item For $P\subseteq \R^{d-1}$  we 
denote $P\times \{z=\lambda\} = (P\times \R)\cap 
\{(\underline{x}, \lambda)\mid \underline{x}\in \R^{d-1} \}\subseteq \R^d$.
\item For $Q\subseteq \R^d$ we denote $Q\cap \{z=\lambda\} = 
Q\cap \{(\underline{x}, \lambda)\mid \underline{x}\in \R^{d-1} \}
\subseteq \R^d$.
\end{enumerate}
\end{notations}

\begin{rmk}In the proof of the earlier stated Theorem of K. Eto in [E] (see introduction), 
he has asserted that 
$\sP$ is a finite union of rational polytopes, which do not intersect at interior points.
In the following lemma we give a detailed proof of this in Lemma~\ref{l1}~(1). 
\end{rmk}
  
\begin{lemma}\label{l1}Let $${\sP_D} = {C_D \setminus 
(\cup_{u_i\in P_D\cap\mathbb{Z}^{d-1}} ((u_i,1) + C_D))}.$$ 
 Then 
\begin{enumerate}
\item ${\overline{\mathcal{P}}_D}$ is a finite 
union of 
rational polytopes $P_1, P_2, \ldots , P_s$ containing the origin such that $P_i \cap 
P_{j}$ is a rational polytope of dimension $< d$ if $i\neq j$. 
Moreover 
\item 
\begin{enumerate}
\item $\dim~(\partial(P_j)\cap \{z= a\}) < d-1$, where 
for a closed set $A\subseteq \R^d$, the set $\partial(A)$ denotes its 
boundary.
\item $\dim(P_i\cap P_j\cap \{z= a\}) < d-1$, for any $a\in \mathbb{R}$.
\end{enumerate}
\end{enumerate}
\end{lemma}
\begin{proof} 
For $d-1=1$, the toric pair $(X, D) = (\P^1,\sO_{\P^1}(n))$, for 
some integer $n\geq 1$. Therefore the lemma is obvious 
from Example~\ref{ep1}. Henceforth we can assume that $d-1\geq 2$.

Let $P_D\subset \R^{d-1}$ be the  convex polytope of dimension $d-1$ 
associated to the pair $(X, D)$. 
Without loss of generality we assume that $P_D$ has the origin as one of the vertices. 
Let $C_D = \mbox{Cone}~(P_D\times \{z=1\})$.

{\underline {\mbox Part}~(1)}:\quad
Let $S= \{F_j\}_j$ be the set of all subcones of $C_D$ obtained by dividing $C_D$ by the
set of hyperplanes 
$$W_D = \{H_{iu}\mid C_{0i} \in \{d-3~~\mbox{faces of}~~ P_D\},~~ u \in P_D\cap \Z^{d-1}\},$$ 
where 
$$H_{iu} = ~~\mbox{the affine span of}~~\{(v_{ik},1), (u,1),
(\underline{0})\mid v_{ik}\in~~\mbox{the vertex set of}~~C_{0i}\}.$$
Thus, the $F_j$ are the closures of the connected components of $C_D\cup_{H\in W_D}H$.
\vspace{5pt}

\noindent{\bf Claim}\quad For each $F_j \in S$ and for each  $u\in P_D\cap \Z^{d-1}$, 
the set $F_j\cap [(u,1)+C_D]^c$ is  convex.

We assume the claim for the moment.

Now we have $$\sP_D = C_D \setminus 
\cup_{u_i\in P_D\cap\mathbb{Z}^{d-1}} ((u_i,1) + C_D) = \cup_jF_j
\setminus
\left\{\cup_{u_i\in P_D\cap\mathbb{Z}^{d-1}} ((u_i,1) + C_D)\right\}.$$ 
  Hence, Part~(1) of the lemma follows by taking  
$$P_j= {\overline{F_j
\setminus \cup_{u_i\in P_D\cap\mathbb{Z}^{d-1}} ((u_i,1) + C_D)}} = 
{\overline{\cap_{u_i\in P_D\cap\mathbb{Z}^{d-1}} F_j
\setminus ((u_i,1) + C_D)}}.$$

\vspace{5pt}

\noindent{\underline{Proof of the claim}:\quad First we prove that for  
given $F_j \in S$ and 
$u\in P_D\cap \Z^{d-1}$, there is a facet $C_i$ of 
$P_D$ such that 
$F_j \subset C^d(C_i, u)$, where
$C^d(C_i,u)$ is the 
 cone generated by $(\underline{0})$, $(u,1)$ and all $(v,1)$, where 
$v$ is a vertex of $C_i$.

Consider the set $\{C^d(C_i,u)\mid C_i~~\mbox{is a facet of}~~P_D,~~ u \notin C_i\}$, 
so that, by construction,  $C^d(C_i,u)$ is a $d$-dimensional cone.
The facets of  any such $C^d(C_i,u)$, other than $C^{d-1}(C_i) = \mbox{Cone over}~C_i$,
are given by the set 
$\{H(C_{ij})\cap C^d(C_i,u)\mid C_{ij} \in~~\{\mbox{facets of}~~C_i\}\},$
where $$H(C_{ij}) =~~\mbox{the affine span of}~~ 
\{(v,1), (\underline{0}), (u,1)\mid v\in~~\mbox{vertex set of}~ C_{ij}\}$$  
are hyperplanes. Since any such $H_{ij}\in W_D$, any such cone $C^d(C_i, u)$ 
is a union of some subset of $S$.
On the other hand note that, for a given $u\in P_D$ we have 
 $C_D = \bigcup_i(C^d(C_i,u))$, where $C_i$ are the facets of $P_D$, and 
the interiors of the $C^d(C_i, u)$ are disjoint.

Hence given $F_j \in S$ and $u\in P_D\cap \Z^{d-1}$ there is a facet $C_i$ of 
$P_D$ such that 
$F_j\subseteq C^d(C_i,u)$.

Now we prove  
the convexity of the set  $F_j\cap [(u,1)+C_D]^c$.

Fix a facet $C_i$ with $F_j\subseteq C^d(C_i,u)$.
Let $x, y\in F_j\cap [(u,1)+C_D]^c$. Then
 $x, y \in C^d(C_i,u)\cap [(u,1)+C_D]^c$, 
therefore we must have expressions  
$$x= \alpha_1(u,1) + c_1~~\mbox{and}~~y= \alpha_2(u,1) + c_2,~~\mbox{where}~~
c_1, c_2\in C^{d-1}(C_i)~~\mbox{and}~~ 
0\leq \alpha_1, \alpha_2 <1.$$ 
This implies that if 
$z$ is any point in the line segment joining $x$ and $y$ then
 $z= l_0(u,1)+c_3$, where $0\leq l_0<1$ and $c_3 \in C^{d-1}(C_i)$.

Since $F_j$ is convex, $z\in F_j$. So we need to prove that 
$z\in [(u,1)+C_D]^c$.

Suppose $z\in (u,1)+C_D$. Then we have 
$z= (u,1)+c$, where $c\in C_D$. This implies 
$$(1-l_0)(u,1)+c= c_3\in C^{d-1}(C_i)\cap [(1-l_0)(u,1)+C_D].$$ 
Now $C^d(C_i,u)$ is a $d$-dimensional cone, which implies $(u, 1)\not\in 
C^{d-1}(C_i)$. Moreover $C^{d-1}(C_i)$ is a facet of $C_D$.
Hence we have a contradiction by the claim given below.
Therefore  we deduce that $z\in [(u,1)+C_D]^c$. This proves that $z
\in  F_j\cap [(u,1)+C_D]^c$. 

Now the convexity of the set $F_j\cap [(u,1)+C_D]^c$ follows from 
Lemma~\ref{sl} given below.

\vspace{5pt}

\noindent \underline {\mbox Part}~(2)}:\quad 
If $C_1$ and $C_2$ are sets in $\R^d$ then
$\partial(C_1\cap C_2) \subseteq \partial(C_1)\cup \partial(C_2)$, (where
 $\partial(C)$ denotes the boundary of $C$). 
Therefore for $P_j$ as above, we have
$$\partial(P_j) \subseteq   \partial(F_j)
\cup_{u_i\in P_D\cap\mathbb{Z}^{d-1}} \partial({\overline{(u_i,1) + C_D)^c}}) 
\subseteq \partial(F_j)
\cup_{u_i\in P_D\cap\Z^{d-1}} \partial((u_i,1) + C_D).$$ 
Therefore
$$\partial(P_j)\subseteq \mbox{facet of}~~(F_j)
\cup_{u_i\in P_D\cap\Z^{d-1}}~~~\mbox{facet of}~~((u_i,1) + C_D).$$
We note that any facet of $(u_i,1) + C_D$ is a translate of a facet of $C_D$ by the point 
$(u_i,1)$. On the other hand any facet of $F_j$ is a subset of an element of 
$W_D$, where the set $W_D$ is defined as in (\ref{e1}) above.
In particular for any facet  $F$ from these set of facets, we have
$\dim~(F\cap \{z = a\}) <d-1$, for any $a\in \R$.
This proves part~(2)(a). 
Part~(2)(b) follows from (a), as for $i\neq j$, the convex polytopes $P_i$ 
and $P_j$ intersects only at their boundary. This completes the proof of   
the lemma.
\end{proof}

\vspace{5pt}

\begin{lemma}\label{sl}Let ${\tilde u} = (u,1)\in C_D$ such that 
${\tilde u}\not\in F$, where $F$ is a facet of $C_D$. Then for 
any $\epsilon >0$, we have
$[\epsilon{\tilde u} + C_D]\cap F = \phi$.\end{lemma}
\begin{proof}
Note that $F= H\cap C_D$, for some hyperplane 
$H = \{\underline{x}= (x_1,\ldots, x_d)\in \R^d\mid \sum a_ix_i = 0, ~~~a_i\in \R\}$. 
Without loss of generality, we assume that   $C_D\subseteq H_+ = 
\{\underline{x}\in \R^d\mid \sum a_ix_i\geq 0\}$.
Therefore for any $\underline{m} = (m_1, \ldots, m_d)\in C_D$ we have 
$\sum_ia_im_i \geq 0$. Moreover, since $\tilde{u} \not\in F$, we have 
$\sum_ia_iu_i >0$.
This imples, for $\epsilon({\tilde u})+\underline{m} = 
(\epsilon({u_1})+{m_1}, \ldots,  \epsilon({u_d})+{m_d})$,
we have $\sum_ia_i(\epsilon({u_i})+{m_i}) >0$. Hence 
$\epsilon({\tilde u})+\underline{m}\in H_+\setminus H\subset F^c$. 
In particular $[\epsilon({\tilde u}) +C_D]\cap F=\phi$.
This proves the lemma.\end{proof}

\vspace{5pt}

Now we are ready to prove Theorem \ref{t2}.
\begin{proof}[{\underline{Proof of Theorem}} \ref{t2}]\quad We note that 
$$n\mathcal{P}_D=\{p\in C_D\ |\ p\notin n(u_i,1) +C_D , \text{for all}\ i=1,\ldots,
r\}.$$ 
Let ${S'}$ be the normalization of the monoid $S$. Hence $K[S']$ is the 
integral closure of $K[S]$ (Theorem 4.39, [BG]). Hence there exists 
$N_0\in \Z$ such that $K[S]_n=K[S']_n$ for all $n\geq N_0$ 
(by Exercise 5.14, [Har]). Hence, for every $\lambda \in \R$,  
there exists 
$n_\lambda\in \N$  such that for all  $n\geq n_\lambda$, we have  
$$\ell_{K[S]}\left(\frac{K[S]}{(Y_1^n,\ldots,Y_r^n)}
\right)_{\lfloor n\lambda\rfloor}=\ell_{K[S']}
\left(\frac{K[S']}{(Y_1^n,\ldots,Y_r^n)}\right)_{\lfloor n\lambda\rfloor}.$$ 

Since $C_D\cap\Z^d =S'$ (by Proposition 2.22, [BG]),
$$n\mathcal{P}_D\cap\Z^d =\{p\in S'\ \arrowvert\ p\notin n(u_i, 1) + C_D,
~~~\mbox{for every}~~~u_i\}.$$

Thus for $n\geq n_{\lambda}$,
$$\ell_{K[S]}\left(\frac{K[S]}{(Y_1^n,\ldots,Y_r^n)}\right)_{\lfloor n\lambda\rfloor}
= \#|(n\mathcal{P}_D\cap \{z=\lfloor n\lambda\rfloor\}|.$$
We denote
$$i(\mathcal{P}_D, n) = \#|n\mathcal{P}_D\cap  \mathbb{Z}^d|~~ \mbox{and}~~  
i(\mathcal{P}_D, n, m) = \#|(nP_D\cap \{z=m\})\cap \Z^{d-1}|.$$

By Lemma \ref{l1}, we have 
$\overline{\mathcal{P}}_D = P_1\cup P_2\cup\cdots \cup P_s$, where
 $P_1, P_2, \ldots , P_s$ are convex rational polytopes 
 such that 
$\dim~(P_i \cap  P_{j}\cap \{z = a\})< d-1$ and $\dim~(\partial (P_j)\cap \{z=a\}) < d-1$, 
for every $a\in \R$.

\vspace{5pt}
\noindent{\bf Claim}\quad If $Q$ is a $d$-dimensional convex polytope then 
for given $\lambda = m_0/q_0$, where  $q_0=p^{n_0}$, for some $n_0\geq 1$ and $q =p^n$,
 we have, 
$$\lim_{q\to\infty}\frac{i(Q,q, 
\lfloor q\lambda\rfloor) }{q^{d-1}}= \mbox{Vol}_{d-1}(Q\cap\{z= \lambda\}).$$

\vspace{5pt}
\noindent{\underline{Proof of the claim}}:\quad
Let  $q =p^n$, where $n\geq n_0$. Note that  we have 
 $\lfloor q\lambda\rfloor = qm_0/q_0$. 
Therefore 
$$ i(Q,q, \lfloor q\lambda\rfloor) = i(qQ\cap\{z=\frac{qm_0}{q_0}\}) =  
i(Q', \frac{q}{q_0}),$$ 
where $Q' = (q_0Q\cap\{z=m_0\})$. 
Now, by Remark~\ref{St}, 
 $$\lim_{q\to\infty}\frac{i(Q,q, \lfloor q\lambda\rfloor) }{q^{d-1}}
= \lim_{q\to\infty}\frac{i(Q',q/q_0)}{q^{d-1}}
= \frac{\mbox{Vol}_{d-1}(Q')}{q_0^{d-1}} = \mbox{Vol}_{d-1}(Q\cap\{z= \lambda\}).$$
This proves the claim.

Let $P_{\leq j}  = P_1 
\cup\ldots \cup P_j$ for  $1\leq j \leq s$. Then
$$i(P_{\leq j_0} , q, {\lfloor q\lambda\rfloor})= i(P_{\leq j_0-1}, q, 
{\lfloor q\lambda\rfloor}) + i(P_{j_0} , q, {\lfloor q\lambda\rfloor}) - 
i([P_{\leq j_0-1} \cap P_{j_0}], q, \lfloor q\lambda\rfloor ).$$
Now 
$$i([P_{\leq j_0-1}\cap P_{j_0}], q, \lfloor q\lambda\rfloor)
= i\left(\frac{q}{q_0}\left[q_0(P_{\leq j_0-1}\cap P_{j_0})\cap 
\{z=m_0\}\right]\right).$$
Therefore, by Lemma~\ref{l1}
$$\lim_{q\to\infty}\frac{i( [P_{\leq j_0-1}\cap P_{j_0}], q,
 \lfloor q\lambda\rfloor)}{
q^{d-1}} = \mbox{Vol}_{d-1}([P_{\leq j_0-1}\cap P_{j_0}]\cap\{z=\frac{m_0}{q_0}\}) = 0.$$
Therefore, by Theorem~1.1 of [T2], we have 
$$HKd(X, D)(\lambda) = 
\lim_{n\to\infty}f_n(\lambda) =\lim_{n\to\infty}\frac{i({\mathcal{P}}_D,q,\lfloor 
q\lambda\rfloor)}{q^{d-1}}
= \lim_{n\to\infty}\sum \frac{i(P_j ,n,\lfloor q\lambda\rfloor)}{q^{d-1}}$$
$$= \sum_j \text{Vol}(P_j\cap \{z=\frac{m_0}{q_0}\}) = \sum_{j=1}^s\phi_{P_j},$$
 where, for the $d$ dimensional polytope $P_j$, the function  
$\phi_{P_j}:(-\infty, \infty) \longto (-\infty, \infty)$ is the
sectional volume function, given by
$t\mapsto \mbox{Vol}_{d-1}(P_j\cap \{z= t\})$.

Note that, by Theorem~1.1 of [T2], $HKd(X, D)$ is a continuous function.
and by Theorem~\ref{thvol}, the function 
$\sum_{j=1}^s\phi_{P_j}$, is also continuous. Since both $HKd(X,D)$ and 
$\sum_{j=1}^s\phi_{P_j}$
 agree on the dense subset 
$\{m/q\mid m\in \Z_{\geq 0}, q=p^n, n\in\Z_{\geq 0}\}\subset \R_{\geq 0}$. 
we conclude that, for every $\lambda \in \R$,  
$$HKd(X, D)(\lambda)  = \sum_j \text{Vol}_{d-1}(P_j\cap 
\{z=\lambda\})=\text{Vol}_{d-1}(\overline{\mathcal{P}}_D\cap \{z=\lambda\}),$$
where the last equality follows from part~(2) of Lemma~\ref{l1}. This completes the proof of the theorem.
\end{proof}

\begin{rmk}
For $\lambda\in\Q_{\geq 0}$, we remark that a generalised (in the sense of Conca 
[Co]) HK density function exists. 
Define 
$$\hat{f}_n(\lambda)=\ell_{K[S]}\left(\frac{K[S]}{(Y_1^n,\ldots,
Y_r^n)}\right)_{\lfloor n\lambda\rfloor}.$$ 
\noindent{\bf Claim}\quad If
$\lambda\in \Q_{\geq 0}$ then 
$HKd(X, D)(\lambda) = \lim_{n\to \infty}\hat{f}_n(\lambda)$.

\vspace{5pt}

\noindent{\underline{Proof of the claim}}:\quad 
Enough to prove that for $\lambda\in \Q_{\geq 0}$, 
 the sequence $\{\hat{f}_n(\lambda)\}$ 
(which contains $\{f_n(\lambda)\}$ as a subsequence) converges. 
Suppose $\lambda=r/s$ with $r\in \Z_{\geq 0}, s\in\Z_{>0}, (r, s)=1$. 
For $n\in \N$, by division algorithm we write $n=l_ns+s_n,$ for $l_n\in\N, 
0\leq s_n <s.$ Write $r_n=\lfloor s_n\frac{r}{s}\rfloor$. 
Then $\lfloor n\frac{r}{s}\rfloor=l_nr+r_n.$ Write 
$$Q_{jn}=\frac{l_ns+s_n}{l_nr+r_n}P_j\cap\{z=1\}~~~\mbox{and}~~~ 
Q_{j0}=\frac{s}{r}P_j\cap\{z=1\}$$ 
for each $P_j$ as in the proof above. 
Then 
$$\displaystyle{nP_j\cap\{z=\lfloor n\lambda\rfloor\}=
(l_nr+r_n)\left(\frac{(l_ns+s_n)}{(l_nr+r_n)}P_j\cap
\{z=1\}\right)=(l_nr+r_n)Q_{jn}}.$$ 
Note that $Q_{jn}\supseteq Q_{j0},$ since $P_j$ contains 
${\underline 0}\in\R^d$ and $\frac{l_ns+s_n}{l_nr+r_n}\geq \frac{s}{r}$. Hence 

$$\lim_{n\to\infty}\frac{i(P_j, n, 
{\lfloor n\lambda\rfloor})}{n^{d-1}} = \lim_{n\to\infty}
\frac{i((l_nr+r_n)Q_{jn}, 1)}{(l_ns+s_n)^{d-1}}
 = \lim_{n\to\infty}\frac{i(Q_{jn}, l_nr+r_n)}{(l_ns+s_n)^{d-1}}$$
$$ \geq \lim_{n\to\infty}\frac{i(Q_{j0}, l_nr+r_n)}{(l_ns+s_n)^{d-1}}
= \left(\frac{r}{s}\right)^{d-1}Vol_{d-1}(Q_{j0})
 = \mbox{Vol}_{d-1}(P_j\cap\{z=\frac{r}{s}\}).$$

Now for each $m\in \N$, for $n\gg 0$, we have $({l_ns+s_n})/({l_nr+r_n})
\leq ({s})({r})+({1})({m})$. As before $Q_{jn}\subseteq 
(\frac{s}{r}+\frac{1}{m})P_j\cap\{z=1\}$. Hence
$$\lim_{n\to\infty}\frac{i(P_j, n, 
{\lfloor n\lambda\rfloor})}{n^{d-1}} = \lim_{n\to\infty}
\frac{i((l_nr+r_n)Q_{jn}, 1)}{(l_ns+s_n)^{d-1}}$$
$$\leq \lim_{n\to\infty}\frac{i((\frac{s}{r}+
\frac{1}{m})P_j\cap\{z=1\}, l_nr+r_n)}{(l_ns+s_n)^{d-1}}
=\left(\frac{r}{s}\right)^{d-1}Vol_{d-1}
\left(\left(\frac{s}{r}+\frac{1}{m}\right)P_j\cap\{z=1\}\right)$$
$$ = \left(\frac{r}{s}\right)^{d-1}\left(\frac{s}{r}+
\frac{1}{m}\right)^{d-1}Vol_{d-1}\left(P_j\cap\big\{z=
\frac{1}{(\frac{s}{r}+\frac{1}{m})}\big\}\right).$$
Letting $m\to \infty$, we see that 
$\lim_{n\to\infty}\frac{i(P_j, n, {\lfloor n\lambda\rfloor})}{n^{d-1}}\leq 
Vol_{d-1}(P_j\cap\{z=\frac{r}{s}\}).$ Thus 
for $\lambda\in\Q_{\geq 0}$, 
$$\lim_{n\to\infty}\frac{i(P_j, n, {\lfloor n\lambda\rfloor})}{n^{d-1}}= 
Vol_{d-1}(P_j\cap\{z=\lambda\}).$$ 
The claim follows easily from previous observations in the proof of the theorem.
\end{rmk}

\section{integral convex polytope and density function}
Let $P_D$ be a convex integral polytope in $\R^{d-1}$ of dimension $d-1$, without 
loss of generality we can assume one of the vertex is $\underline{0}\in \R^{d-1}$. 
\begin{notations}\label{n1}\begin{enumerate}
\item Let 
$\{v_1, \ldots, v_l\}\subset \R^{d-1}$  be the set of vertices of $P_D$.
\item Let $C_D\subseteq \R^d$ be the cone given by 
$\{(v_i, 1)\}$ and the origin $\underline{0}$ of $\R^d$. 
\item Let $\sP_D = C_D\setminus 
\cup_{u\in P_D\cap \Z^{d-1}}((u, 1)+C_D)$.
Similarly, for an integer $m\geq 1$, $\sP_{mD} = C_{mD}\setminus 
\cup_{u\in P_{mD}\cap \Z^{d-1}}((u, 1)+C_{mD})$.
\item 
Let $W_0\subset \R^{d-1}$ be the unit cell $[0,1]^{d-1}$. For a point $v\in \Z^{d-1}$, 
the set $W_v$ denotes the cell which is the translate of $W_0$ by $v$, that is 
$W_v = v+W_0$.
\item Let $l_D:=\mbox{min}\{t\in \mathbb{R}_{\geq 0}\ |\ B(0, t)\supset P_D\}$, where 
$B(0, t)$ is the closed ball of radius $t$ at $0$.
\end{enumerate}

\end{notations}

\begin{rmk}Note that $$C_{mD}\cap \{z=\lambda+1\} = P_{(1+\lambda)mD}\times 
\{z=1+\lambda\} = \{(\sum_ib_iv_i, 1+\lambda)\in\R^d \mid 
b_i\geq 0, \sum_ib_i = m+m\lambda\}.$$
\end{rmk}

\begin{lemma}\label{l3}For an integer $m\geq 1$, where $l$ is the number 
of vertices of $P_D$, we have 
\begin{enumerate}
\item $\sP_{mD}\cap\{z= 1+\lambda\} = \phi$,
for all $\lambda \geq l/m$.
\item In particular
\begin{enumerate}
\item $\sP_{mD}\cap\{z= 1+\lambda\} \subseteq P_{(m+l)D}\times \{z = 1+\lambda\}$, for $\lambda \geq 0$ and 
 
 \item $\sP_{mD}\cap\{z\in [1, \infty)\} \subseteq P_{(m+l)D}\times \{z\in [1, 1+l/m]\}$.
\end{enumerate}
\end{enumerate}
\end{lemma}
\begin{proof} We assume the following claim for the moment.

\vspace{5pt}

\noindent{\bf Claim}\quad $C_{mD}\cap \{z= 1+l/m\} \subseteq 
\bigcup_{u\in P_{mD}\cap \Z^{d-1}}((u,1)+C_{mD})$.

\vspace{5pt}

\noindent{(1)}\quad Note for  any $\lambda \geq l/m$ and for 
$(w', 1+\lambda)\in C_{mD}\cap \{z=1+\lambda\}$, there exists  $w_0\in P_{mD}$ 
such that  $w' = ((1+\lambda)w_0, 1+\lambda)$. Therefore we  can write
$$(w', 1+\lambda) = ((1+\frac{l}{m})w_0, 1+\frac{l}{m})+((\lambda-\frac{l}{m})w_0, 
(\lambda-\frac{l}{m})) \in 
((1+\frac{l}{m})w_0, 1+\frac{l}{m})+C_{mD}.$$ Hence by the above claim, 
$C_{mD}\cap \{z= 1+\lambda\} \subseteq 
\bigcup_{u\in P_{mD}\cap \Z^{d-1}}((u,1)+C_{mD})$.
Therefore $\sP_{mD}\cap\{z= 1+\lambda\} = \phi$,
for all $\lambda \geq l/m$. This proves the first assertion.

\noindent{(2)}\quad The above claim implies that $$\sP_{mD}\cap \{z\in [1, \infty)\}
= \sP_{mD}\cap \{z\in [1, 1+l/m]\}
\subseteq C_{mD}\cap \{z\in [1, 1+l/m]\}$$
$$\subseteq P_{(1+l/m)mD}\times \{z\in [1, 1+l/m]\}
= P_{(m+l)D}\times \{z\in [1, 1+l/m]\}.$$
Note that the last inequality holds as $\underline{0}\in P_D$ implies 
$P_{\lambda D} \subseteq P_{\lambda'D}$, if $\lambda'\geq \lambda$.
This  proves both the parts of the second assertion.

Now we give a

\vspace{5pt}
 
\noindent{\underline{Proof of the claim}}:\quad Let $(w, 1+l/m)\in C_{mD}\cap \{z= 1+l/m\}$.
Then $w= \sum_ia_imv_i$, where $a_i\geq 0$ in $\R$ and $\sum_ia_i = 1+l/m$.
We  write $ma_i = [ma_i]+\{ma_i\}$, where $[x]$ denote 
the integral part of a number $x$ and $\{x\} = x-[x]$ denotes the fractional part of 
$x$.
 Since $0\leq \sum_i\{ma_i\} <l$ and $\sum_i ma_i = m+l$, we have 
$m+l\geq \sum_i[ma_i] \geq m$.
Therefore we can choose nonnegative integers $c_1, \ldots, c_l$ such that 
$c_i \leq [ma_i]$ and $\sum_i([ma_i]-c_i) =m$. Now we can write 
$$(w, 1+\frac{l}{m}) = (\sum_i([ma_i]-c_i)v_i, 1) + (\sum_i\frac{(\{ma_i\}+c_i)}{m}
mv_i, \frac{l}{m})
\in (P_{mD}\cap \Z^{d-1},1)+ C_{mD}.$$
This proves the claim and hence the lemma.

\end{proof} 

\begin{lemma}\label{l2}Let $W_v\subset \R^{d-1}$ be a unit cell as given in 
Notations~\ref{n1}~(4) above.
Then, for any fixed integer $m\geq 1$, 

$$\mbox{Vol}_{d-1}[ (W_v\times\{z=\lambda+1\})\setminus
\bigcup_{u\in\mathbb{Z}^{d-1}}((u,1)+C_{mD})]=\mbox{Vol}_{d-1}
[(W_0\times \{z=m\lambda+1\})\setminus
\bigcup_{u\in\mathbb{Z}^{d-1}}((u,1)+C_{D})].$$
\end{lemma}
\begin{proof}
Define $$\psi : (W_v\times \{z=\lambda+1\})\setminus
\bigcup_{u\in\mathbb{Z}^{d-1}}((u,1)+C_{mD})\longrightarrow 
(W_0\times\{z=m\lambda+1\})\setminus\bigcup_{u\in\mathbb{Z}^{d-1}}((u,1)+C_{D})$$ 
by $(x, \lambda+1)\mapsto (x -v, m\lambda+1)$, for $x\in W$.
 Note that $(x-v, m\lambda+1)\in (u,1)+C_D$ if and only if  $(x, \lambda+1)\in 
(u+v,1)+C_{mD}$. Thus $\psi$ is a well defined isometry. 
\end{proof}

\begin{defn}\label{d1}Let $W_v$ 
be a  $d-1$ dimensional unit cell, for some $v\in \Z^{d-1}$.
We define 
 the sets, for $0 \leq \lambda$,  
$$\Phi^v_{mD}(\lambda) = (W_v\times 
\{z=\lambda+1\})\setminus\bigcup_{u\in P_{mD}\cap\mathbb{Z}^{d-1}}((u,1)+C_{mD}).$$
and 
$$\Psi^v_{mD}(\lambda) = (W_v\times 
\{z=\lambda+1\})\setminus\bigcup_{u\in \mathbb{Z}^{d-1}}((u,1)+C_{mD}).$$
Let 
$\varphi^v_{mD}(\lambda):[0, \infty)\longrightarrow [0, 1]$ be the function
$\varphi^v_{mD}(\lambda) = \mbox{Vol}_{d-1}(\Phi^v_{mD}(\lambda))$
and similarly let 
$\varphi_{mD}(\lambda):[0, \infty)\longrightarrow [0, 1]$ be the function 
$\varphi_{mD}(\lambda) = \text{Vol}_{d-1}(\Psi^v_{mD}(\lambda))$.
\end{defn}

\begin{rmk}\label{***}
\begin{enumerate}
\item By Lemma~\ref{l2}, the $\varphi_{mD}$ is well defined 
(independent of choice of $v$ in $\Z^{d-1}$). Also
\item  
by Lemma~\ref{l2}, $\varphi_{mD}(\lambda) = \varphi_D(m\lambda)$.
\item 
By definition it follows that $\varphi_{mD}(\lambda)\leq \varphi^v_{mD}(\lambda) \leq 1$.
\end{enumerate}
\end{rmk}

\begin{defn}\label{d3}
\begin{enumerate}
\item Let $r\in \R_{\geq 1}$ such that $P_{rD}$ contains a unit cell $W_v$, for 
some $v$.
\item
 For a unit cell $W_v$,
  let 
$$l(W_v)=\{u\in \mathbb{Z}^{d-1} \ | \ 
{\tilde d}(u, w)\leq (lr)l_D +\sqrt{d-1},~~~\mbox{for all}~~w\in W_v \},$$ 
where 
${\tilde d}$ denotes the Euclidean metric on $\mathbb{R}^{d-1}$ and $l$ denotes 
the number of vertices of $P_D$. 
\end{enumerate}
\end{defn}

\begin{lemma}\label{l4} For any given integer  $m\geq 1$ and for $\lambda\geq 0$, 
we have  $\Psi^v_{mD}(\lambda) = A^v_m(\lambda)$, where
$$A^v_m(\lambda) = (W_v\times\{z=\lambda+1\})\setminus
\bigcup_{u\in l(W_v)}((u,1)+C_{mD}).$$
Moreover, for $\lambda \geq l r/m$,
 $$\Psi^v_{mD}(\lambda) = A^v_m(\lambda) = \emptyset.$$  
In particular  $\varphi_{mD}$ is a compactly supported continuous 
function.
\end{lemma}
\begin{proof}Let $0\leq \lambda \leq lr/m$. 
Note that $\Psi^v_{mD}(\lambda) \subseteq  A^v_m(\lambda)$, for every $\lambda$. 
Therefore, it is enough to show that 
$$(W_v\times\{z=\lambda+1\})\bigcap\bigcup_{u\in\Z^{d-1}}((u,1)+C_{mD})
\subseteq(W_v\times\{z=\lambda+1\})\bigcap\bigcup_{u\in l(W_v)}((u,1)+C_{mD}).$$
Suppose, there is  $x=(u, 1)+\sum b_i(mv_i, 1)\in W_v\times\{z=\lambda+1\}$, 
for some $u\in \Z^{d-1}\setminus l(W_v)$, and $x=(w', \lambda+1)$, for some $w'\in W_v$
 and
$\sum b_i=\lambda$. Therefore, there exists $w\in W_v$, such that 
$$lr\cdot l_D+\sqrt{d-1} \ <  {\tilde d}(u,w)\leq 
{\tilde d}(u, w')+ {\tilde d}(w', w)\leq {\tilde d}(0,
\sum b_imv_i)+ \sqrt{d-1},$$ 
as  ${\tilde d}(w,w') \leq \sqrt{d-1}$, for any $w, w'\in W_v$ and 
${\tilde d}(u, w') =  {\tilde d}(u, u+\sum_ib_imv_i) = {\tilde d}(0, \sum b_imv_i)$. 
This implies,
$l_D < {\tilde d}(0,
\sum (b_im/lr)v_i)$. Therefore,
 by the definition of $l_D$ (as $\underline{0}\in P_D$),
$$\sum (b_im/lr)v_i\notin P_D  \implies \sum b_im/lr > 1\implies
\lambda > lr/m,$$  which is
a contradiction. Hence $\Psi^v_{mD}(\lambda) = A_m^v(\lambda)$, for $\lambda \leq lr/m$.

\vspace{5pt}
\noindent{Let} $\lambda \geq  lr/m$.

\noindent{\bf Claim}\quad $\Psi^v_{mD}(\lambda) = \emptyset $.

\noindent{\underline{Proof of the claim}}:\quad Let $x = (w, \lambda+1)
\in \Psi^v_{mD}(\lambda)$. Then 
 there is $v'\in \Z^{d-1}$ such that $w-v'\in P_{rD}$ (as 
$P_{rD}$ contains a unit cell).
We can write  $(w,\lambda+1) = (w-v', \lambda)+(v',1)$. 
Now $w-v'\in P_{rD}\subseteq P_{\lambda mD}$ (as $r/m\leq rl/m \leq \lambda$).
This implies $(w-v', \lambda)\in P_{\lambda mD}\times \{z= \lambda\}\subseteq C_{mD}$.
Hence $(w, \lambda+1)\in (v',1)+C_{mD}$, where $v'\in \Z^{d-1}$. This implies 
$\Psi^v_{mD}(\lambda) = \emptyset $, for $\lambda \geq rl/m$.

Now we have  $A_m^v(lr/m) = \Psi^v_{mD}(lr/m) = \emptyset$.
Let $x= (w, \lambda+1)\in A_m^v(\lambda)$, for some 
$\lambda \geq lr/m$. Then $(w, 1+lr/m)\in (u,1)+C_{mD}$, for some $u\in l(W_v)$. 
Therefore 
$$(w, 1+\lambda) = (w, 1+lr/m) + (\underline{0}, \lambda-lr/m) \in (u,1)+C_{mD},
~~\mbox{where}~~u\in l(W_v)$$ and
$\{\underline{0}\times \R_{\geq 0}\}\subseteq C_{mD}$.
Hence $A_m^v(\lambda) = \phi$, for 
$\lambda\geq lr/m$. 

Since $l(W_v)$ is a finite set, the function $\varphi_{mD}$ is 
continuous. This completes the proof of the lemma.

\end{proof}

\begin{rmk}\label{r3}
Morever if we write $P_{kD}\cap \Z^{d-1} = P'_{kD}\cup P''_{kD}$ 
such that  ($l(W_v)$ is given as in Definition~\ref{d3})
$$P_{kD}'=\{v\in P_{kD}\cap \mathbb{Z}^{d-1} \ \arrowvert\ l(W_v)\subset P_{kD}\}$$ 
 and 
$$P_{kD}''=\{v\in P_{kD}\cap \mathbb{Z}^{d-1}\ \arrowvert \ l(W_v) \nsubseteq P_{kD}\}$$
then  
$$v\in P'_{kD}
 \implies
 l(W_v)\subseteq P_{kD}\cap \Z^{d-1}\subseteq \Z^{d-1} \implies 
\varphi^v_{kD} = \varphi_{kD},~~\mbox{by}~~Lemma~\ref{l4}.$$
\end{rmk}

\begin{notations} For any given two closed sets $Q'$ and $Q''$ in 
$\R^{d-1}$, let 
 $d(Q', Q'') = \min\{d(x, y)\mid 
x\in Q', y\in Q''\}$ denote the distance between the sets
$Q'$ and $Q''$. \end{notations}

\begin{lemma}\label{ln}Let $Q$ be a convex $d-1$-dimensional rational polytope 
in $\R^{d-1}$  such that it contains the origin in its interior.
Then, there is a constant $\delta_0 > 0$, depending on $Q$, such that for every rational 
$m {\geq 1}$ 
and for every integer $l\geq 0$, we have
$$d(\partial(mQ),  \partial((m+l)Q) \geq l\delta_0,$$
where $\partial(Q)$ denotes  the  boundary of $Q$ in $\R^{d-1}$.
\end{lemma}
\begin{proof} 
Let $\{F_i\}_i$ be the set of facets of $Q$; then  $\{mF_i\}_i$ is the 
set of  facets of 
$mQ$, for any rational number $m\geq 1$. Moreover, if $F_i = H_i\cap Q$ then 
$mF_i = mH_i\cap mQ$, where $H_i$ denotes the supporting hyperplane of $Q$ at $F_i$.
Now, since $\partial(mQ)$ and $\partial((m+l)Q)$ are compact closed sets, there exist
$x_0\in \partial(mQ)$ and $y_0\in \partial((m+l)Q)$ such that 
$d(\partial(mQ),  \partial((m+l)Q) = d(x_0, y_0)$.

\vspace{5pt}

\noindent{\bf Claim}\quad $d(x_0, y_0)\geq \min\{d(mH_i, (m+l)H_i)
\mid H_i \in~\{\mbox{Supporting hyperplanes of}~~Q\}\}.$ 

\vspace{5pt}

\noindent\underline{Proof of the claim}:

\noindent{\underline{Case}}~1\quad If $x_0\in mF_i$ and $y_0\in (m+l)F_i$, for some 
facet $F_i$ 
of $Q$ then
$d(x_0, y_0) \geq  d(mH_i, (m+l)H_i)$, as
$mF_i\subset mH_i$ and $(m+l)F_i\subset (m+l)H_i$, where $mH_i$ and $(m+l)H_i$ 
are parallel hyperplanes.

\noindent{\underline{Case}}~2\quad Suppose $x_0\in mF_i$ and $y_0\in (m+l)F_j$,
where $i\neq j$, then 
$$d(x_0, y_0) \geq  d(mF_i, (m+l)F_j)\geq d(mQ, (m+l)Q) = d(x_0, y_0).$$ 
As $mH_j$ is a supporting hyperplane for $mQ$, the entire polytope  
$mQ$ lies in one side of $mH_j$, say, $mQ \subset (mH_j)^+$, which implies  
$mF_i\subset (mH_j)^+$.
On the other hand 
$(m+l)F_j \subset (m+1)H_j$. 
Hence $$d(x_0, y_0) = d(mF_i, (m+l)F_j)\geq d((mH_j)^+, (m+l)H_j) =  
d(mH_j, (m+l)H_j).$$
This proves the claim. 

Let $\delta_i = d(H_i, (\underline{0}))$; then $\delta_i = \|x\|$, for some 
$x\in H_i$. Now  it is easy to check that $d(mH_i, (\underline{0})) = \|mx\| = m\|x\|$
and  $d(mH_i, (m+l)H_i) = d(mx, (m+l)x) =  l\delta_i$, where $\delta_i >0$ as 
$H_i$ does not pass through the origin. 
Since, there are only finitely many facets and hence finitely many $H_i$, 
$\delta_0 = \min\{\delta_i\} >0$.
This proves  the lemma. 
\end{proof}

\begin{lemma}\label{sl1} Let 
$$S_1=\{v\in \Z^{d-1}\setminus P_{kD}\cap \Z^{d-1}\mid W_v\cap P_{kD}\neq \emptyset,~~
W_v\cap (P_{kD})^c\neq \emptyset\}$$
and let $P_{kD}''$ be as in Remark~\ref{r3}. Then 
$\#|P''_{kD}| = O(k^{d-2})$ and $\#|S_1| = O(k^{d-2}).$
\end{lemma}
\begin{proof}Let $l$ be the number of vertices of $P_D$ and let $r\geq 1$ be 
an integer such that the interior of $P_{rD}$ contains a unit cell $W_{v}$ for 
some $v\in \Z^{d-1}$.
   Let $l_D$ be as in Notations~\ref{n1}. 

Then  $P_{rD}$ contains the  lattice point $v$ in its  interior.
Let $Q= P_{rD}-v$ then $Q$ is a convex (integral) $d-1$ dimensional polytope in 
$\R^{d-1}$ such that the origin is  in the interior of $Q$. 
Let $\delta_0 >0$ be  a constant for $Q$, as given in 
Lemma~\ref{ln}. 

For  ${\tilde l} = (lr)l_D+ {\sqrt{d-1}}$,  
we can choose  $l_1\in \Z_{\geq 0}$ ({\it e.g},  
$l_1\geq {\tilde l}/\delta_0$) so that  
we have  
\begin{equation}\label{d-2} d(\partial(\left(\frac{k}{r}+l_1\right)Q), 
\partial(\left(\frac{k}{r}\right)Q)) \geq {\tilde l}~~\mbox{and}~~
d(\partial(\left(\frac{k}{r}\right)Q), \partial( \left(\frac{k}{r}-l_1\right) Q)) 
\geq {\tilde l}.\end{equation}

Note that  $(\underline{0})\in Q$ implies   
$(k/r-l_1)Q \subseteq (k/r)Q\subseteq  (k/r+l_1)Q$, which is the same as  
$$  P_{(k-l_1r)D}-(k/r-l_1)v \subseteq  
P_{kD}-(k/r)v\subseteq   P_{(k+l_1r)D} - (k/r+l_1)v.$$
Hence 
$P_{(k-l_1r)D}+l_1v\subseteq P_{kD} \subseteq P_{(k+l_1r)D} -l_1v$ 
and, by (\ref{d-2}),
$$d(\partial(P_{kD}), \partial(P_{(k-l_1r)D}+l_1v))\geq {\tilde l}~~\mbox{and}~~
d(\partial(P_{(k+l_1r)D}-l_1v), \partial(P_{kD}))\geq {\tilde l}.$$
(Note that translation by $kv/r$ is an isometry).

Therefore  $v_1\in P_{(k-l_1r)D}+l_1v$ implies that $l(W_{v_1})\subseteq P_{kD}$. Hence 
$P_{kD}''\subseteq P_{kD}\setminus (P_{(k-l_1r)D}+l_1v)$. 
Similarly $v_1\in S_1$ implies that $d(v_1, \partial(P_{kD}))\leq \sqrt{d-1})
\leq {\tilde l}$.
Hence  $S_1 \subseteq (P_{(k+l_1r)D}-l_1v)\setminus P_{kD}.$
Now 
$\#|P''_{kD}| \leq  \#|P_{kD}|-\#|P_{(k-l_1r)}D| = O(k^{d-2})$, and
similarly for $\#|S_1|$. This proves the lemma.\end{proof}

\section{Main theorem}

\vspace{5pt}

First we give a proof of Proposition~\ref{p5}, which replaces $HKd(X, kD)$ by 
$\varphi_{kD}$ upto $O(k^{d-2}).$

\vspace{5pt}

{\underline{Proof of Proposition}~~\ref{p5}}:\quad By Theorem~\ref{t2}, 
for $\lambda \geq 0$, 
$$HKd(X, kD)(\lambda+1)
 = \mbox{Vol}(\mathcal{P}_{kD}\cap \{z=\lambda+1\}).$$
By Lemma~\ref{l3}(2), for $\lambda\geq 0$, 
$$\sP_{kD}\cap \{z=1+\lambda\} = \left[P_{(k+l)D}
\times\{z= 1+\lambda \}\right] \cap \left[\sP_{kD}\cap \{z=1+\lambda\}\right] 
\subseteq Q_0(\lambda)\cup Q_1(\lambda),$$
where
$Q_0(\lambda) = (P_{(k+l)D}\setminus P_{kD})
\times \{z= 1+\lambda \}$ and 
$$Q_1(\lambda) = (P_{kD}
\times \{z= 1+\lambda \})\cap (\sP_{kD}\cap \{z=1+\lambda\}).$$
Now
 one can cover $P_{kD}$ by unit cells as follows:
$P_{kD} \subseteq \cup_{v\in S_1}W_v\cup\cup_{v\in P_{kD}\cap \Z^{d-1}}W_v,$
where 
$$S_1=\{v\in \Z^{d-1}\setminus P_{kD}\cap \Z^{d-1}\mid W_v\cap P_{kD}\neq \emptyset,~~
W_v\cap (P_{kD})^c\neq \emptyset\}.$$
Therefore (see Definition~\ref{d1}) 
$$Q_1(\lambda)\subseteq \cup_{v\in S_1}\Phi^v_{kD}(\lambda)
\cup \cup_{v\in P_{kD}\cap \Z^{d-1}}\Phi^v_{kD}(\lambda).$$
Hence
$$\mbox{Vol}_{d-1}\sP_{kD}\cap\{z= \lambda +1\}
\leq ((k+l)^{d-1}-k^{d-1})\mbox{Vol}(P_D)+
\sum_{v\in S_1}\varphi^v_{kD}(\lambda)+ \sum_{v\in P_{kD}\cap \Z^{d-1}}
\varphi^v_{kD}(\lambda)$$
$$=  O(k^{d-2})+ 
 \sum_{v\in S_1}\varphi^v_{kD}(\lambda)+ \sum_{v\in P_{kD}\cap \Z^{d-1}}
\varphi_{kD}(\lambda) +\sum_{v\in P''_{kD}}
[\varphi^v_{kD}(\lambda)-\varphi_{kD}(\lambda)],$$
where the last equality
 follows as $\varphi^v_{kD} = \varphi_{kD}$, for $v\in P_{kD}'$ (see Remark~\ref{r3}).

On the other hand, for $v\in P'_{kD}$, we have $W_v\subseteq P_{kD}$, therefore 
$\cup_{v\in P'_{kD}}\Phi_{kD}^v(\lambda) \subseteq Q_1(\lambda)$. Hence, 
(note $\dim~(\Phi^v_{kD}(\lambda)\cap\Phi^{v'}_{kD}(\lambda))<d-1$, for $v\neq v'$) 
$$\mbox{Vol}_{d-1}\sP_{kD}\cap\{z= \lambda +1\}\geq \sum_{v\in 
P'_{kD}}{\varphi_{kD}}(\lambda) =
 \sum_{v\in P_{kD}\cap \Z^{d-1}}\varphi_{kD}(\lambda)- \sum_{v\in P''_{kD}}
\varphi_{kD}(\lambda).$$

Also, by definition, $0\leq \varphi^v_{kD}(\lambda), \varphi_{kD}(\lambda) \leq 1$. 
  Now, by Lemma~\ref{sl1},  we can conclude that
 $$\mbox{Vol}_{d-1}\sP_{kD}\cap \{z= \lambda +1\}
=   \sum_{v\in P_{kD}\cap \Z^{d-1}}\varphi_{kD}(\lambda) +
O(k^{d-2})
= h^0(X, \mathcal{O}_X(kD))\varphi_{kD}(\lambda) 
+O(k^{d-2}).$$ This proves the proposition.\hfill$\Box$

\begin{rmk}\label{rm}Let $R = \oplus_{n\geq 0}R_n$ be a standard graded ring over a field $K$.
Let 
$$R^{(k)}=\bigoplus_{d\geq 0} R_{kd}=R_0\oplus R_k\oplus R_{2k}\oplus\cdots~~~\mbox{and}~~~
m_{R^{(k)}}=\bigoplus_{d\geq 1}R_{kd}$$
be the k-fold Veronese ring and its  homogeneous maximal ideal, respectively.  
Recall that we have defined $HKD(X, D)$ (or $e_{HK}(X,D)$) for a toric pair $(X, D)$ as the 
HK density function (or HK multiplicity, respectively) of 
the associated homogeneous coordinate ring with respect to its graded maximal ideal.
With this notation, if  $(R, {\bf m})$ denotes the 
homogeneous coordinate ring with the graded maximal ideal 
${\bf m}$ for a toric pair $(X, D)$ then 
we have $e_{HK}(R, {\bf m}) = e_{HK}(X, D)$ and 
$$e_{HK}(R, {\bf m}^k) = ke_{HK}(R^{(k)}, {\bf m}_{R^{(k)}}) = k e_{HK}(X, kD),$$
$$e_0(R, {\bf m}^k) = ke_0(R^{(k)}, {\bf m}_{R^{(k)}}) = ke_0(X, kD) = k^de_0(X, D).$$
\end{rmk}

Now we give a proof of the main theorem of this paper.
  
\vspace{5pt}

\noindent{\underline{Proof of Theorem~\ref{t1}}}\quad 
We denote the co-ordinate ring $K[S]$ of $(X, D)$\ by $R$. Therefore
$R= K[S] = K[{\chi}^{(u_1,1)},\ldots, {\chi}^{(u_r,1)}]$.

Then, by Remark~\ref{rm}, we have
\begin{equation}\label{e5} e_{HK}(R, m^k)-\frac{e_0(R, m^k)}{d!} = 
 ke_{HK}(X, kD)- k\frac{e_0(X, kD)}{d!}.\end{equation}
But $$\frac{e_0(X, kD)}{d!}= \frac{1}{d}\mbox{Vol}(P_{kD}) = 
\int_0^1\mbox{Vol}_{d-1}(P_{\lambda kD})d\lambda =\int_0^1
\mbox{Vol}_{d-1}(\mathcal{P}_{kD}\cap\{z= \lambda\})d\lambda.$$
Moreover, by Theorem~1.1 of [T2] and by Proposition~\ref{p5}, 
$$e_{HK}(X, kD) = \int_0^\infty HKd(X, kD)(\lambda)d\lambda
= \int_0^\infty\mbox{Vol}_{d-1}(\mathcal{P}_{kD}\cap\{z = \lambda\})d\lambda.$$
Hence 
$$e_{HK}(X, kD)-\frac{e_0(X, kD)}{d!}
= \int_0^{l/k}\mbox{Vol}_{d-1}(\mathcal{P}_{kD}\cap\{z= 1+\lambda\})d\lambda,$$
where the last equality follows by Lemma~\ref{l3}.
By Proposition~\ref{p5}
$$
\int_{0}^{l/k}\mbox{Vol}_{d-1}(\mathcal{P}_{kD}\cap \{z=1+\lambda\})d\lambda 
  = h^0(X, \sO_X(kD)) \int_0^{l/k} 
\varphi_{kD}(\lambda)\ d\lambda\ + O(k^{d-3}),$$
where by Remark~\ref{***}~(2)
$$ \int_0^{l/k} \varphi_{kD}(\lambda)\ d\lambda
= \frac{1}{k}\int_0^{l} \varphi_{D}(\lambda) d\lambda~~~\mbox{
and}~~~h^0(X, \sO_X(kD)) = \frac{e_0(X, D)k^{d-1}}{(d-1)!}
+ O(k^{d-2}).$$
Therefore 
$$\int_{0}^{l/k}\mbox{Vol}_{d-1}(\mathcal{P}_{kD}\cap \{z=1+\lambda\})d\lambda 
= k^{d-2} \frac{e_0(X,D)}{(d-1)!}\int_0^{l} 
\varphi_{D}(\lambda) d\lambda\ +O(k^{d-3}). $$
Therefore, by (\ref{e5}),  we have 
$$\lim_{k\to \infty}\ \frac{1}{k^{d-1}}\left(e_{HK}(R, m^k)-\frac{e_0(R, m^k)}{d!}\right) 
= \frac{e_0(R, {\bf m})}{(d-1)!}\int_0^\infty \varphi_{D}(\lambda)d\lambda.$$
This proves the theorem. $\Box$

\vspace{5pt}

Now we give a proof of Proposition \ref{t3} which shows the multiplicative property 
of the function $\varphi$ on the set of projective toric varieties. 

\vspace{5pt}

\noindent{\underline{Proof of Proposition \ref{t3}}}:\quad 
Let $(X, D)$ and $(Y, D')$ be two toric pairs of dimension $d-1\geq 1$ and 
$d'-1\geq 1$, respectively.
Now, by Proposition~\ref{p5}
$$ HKd(X, kD)(\lambda+1)
= h^0(X, \mathcal{O}_X(kD))\varphi_{X, kD}(\lambda)+O(k^{d-2}),~~\mbox{for}~~
\lambda\geq 0,$$
$$ HKd(Y, kD')(\lambda+1)
= h^0(Y, \mathcal{O}_Y(kD'))\varphi_{Y, kD'}(\lambda)+O(k^{d'-2}), 
~~\mbox{for}~~\lambda\geq 0.$$

Let $$e_X= \frac{e_0(X, D)}{(d-1)!},~~~e_Y= \frac{e_0(Y, D')}{(d'-1)!}~~\mbox{and}~~
e_{X\times Y} =  \frac{e_0(X\times Y, D\boxtimes D')}{(d+d'-2)!}
= \frac{e_0(X, D)}{(d-1)!}\frac{e_0(Y, D')}{(d'-1)!}.$$

Therefore , by Proposition 2.14 (and Definition~2.13) of [T2], we have 
$$\begin{array}{l}
HKd(X\times Y, k(D\boxtimes D'))(\lambda+1)\\\\
=e_X[k(\lambda+1)]^{d-1}HKd(Y,D')(\lambda+1)+
e_Y[k(\lambda+1)]^{d-1}HKd(X,D)(\lambda+1)\\\\
- HKd(X, D)(\lambda+1)HKd(Y,D')(\lambda+1)\\\\
=\left(e_Xk^{d-1}{(\lambda+1)}^{d-1}\right)
\left(h^0(Y, \mathcal{O}_Y(kD'))\varphi_{Y, kD'}(\lambda)+O(k^{d'-2})\right)\\
+\left(e_Yk^{d'-1}{(\lambda+1)}^{d'-1}\right)
\left(h^0(X, \mathcal{O}_X(kD))\varphi_{X, kD}(\lambda)+O(k^{d-2})\right)\\
- \left(h^0(Y, \mathcal{O}_Y(kD))\varphi_{Y, kD'}
(\lambda)+O(k^{d'-2})\right)\left(h^0(X, \mathcal{O}_X(kD))
\varphi_{X, kD}(\lambda)+O(k^{d-2})\right).
\end{array}$$
Since, $\varphi_{X, kD}(\lambda)$ and $\varphi_{Y, kD'}(\lambda) \in [0,1]$ and 
$$h^0(X, \mathcal{O}_X(kD))= e_X k^{d-1}+O(k^{d-2})~~\mbox{and}~~
h^0(Y, \mathcal{O}_Y(kD')) = e_Yk^{d'-1}+O(k^{d'-2}),$$
we have 
$$\begin{array}{l}
HKd(X\times Y, k(D\boxtimes D'))(\lambda+1)\\\\
=\left(e_Xk^{d-1}{(\lambda+1)}^{d-1}\right)
\left(e_Yk^{d'-1}\varphi_{Y, kD'}(\lambda)+O(k^{d'-2})\right)\\\\
+\left(e_Yk^{d'-1}{(\lambda+1)}^{d'-1}\right)
\left(e_Xk^{d-1}\varphi_{X, kD}(\lambda)+O(k^{d-2})\right)\\\\
- \left(e_Yk^{d'-1}\varphi_{Y, kD'}(\lambda)+O(k^{d'-2})\right)
\times\left(e_Xk^{d-1}\varphi_{X, kD}(\lambda)+O(k^{d-2})\right)\\\\
=\left(e_Xe_Yk^{d+d'-2}\right)
\left({(\lambda+1)}^{d-1}\varphi_{Y, kD'}(\lambda)+{(\lambda+1)}^{d'-1}
\varphi_{X, kD}(\lambda)-\varphi_{X, kD}(\lambda)\varphi_{Y, kD'}(\lambda)\right)\\\\
+ \left(e_Xk^{d-1}{(\lambda+1)}^{d-1}\times O(k^{d'-2})\right) 
+\left(e_Yk^{d'-1}{(\lambda+1)}^{d'-1}\times O(k^{d-2})\right)
+ O(k^{d+d'-3}).
\end{array}$$
By Remark~\ref{***}, we have 
 $\varphi_{X, kD}(\lambda) = \varphi_{X,D}(k\lambda)$ and similarly for the 
pair $(Y, D')$.
 In particular for any $x\in \R_{\geq 0}$ and any integer $k\geq 1$, 
we have (by substituting $\lambda = x/k$), 
$$\begin{array}{l}
\left[HKd(X\times Y, k(D\boxtimes D'))({x}/{k}+1)\right]/k^{d+d'-2}\\\\
 = (e_Xe_Y)\left[ ({x}/{k}+1)^{d-1}\varphi_{Y, D'}(x)+({x}/{k}+1)^{d'-1}
\varphi_{X, D}(x)-\varphi_{X, D}(x)\varphi_{Y, D'}(x)\right]\\\\
+ \frac{1}{k}\left[e_X({x}/{k}+1)^{d-1}\times O(1)
+ e_Y({x}/{k}+1)^{d'-1}\times O(1)\right]
+ O(k^{d+d'-3}).\end{array}$$

On the other hand, as $X\times Y$ is a toric variety, 
we have from Proposition \ref{p5},
$$\begin{array}{l}
 HKd(X\times Y, k(D\boxtimes D'))(\lambda+1)\\\\
=\left[h^0(X\times Y, k(D\boxtimes D'))\right]
 \left[\varphi_{X\times Y, k(D\boxtimes D')}(\lambda)\right]+O(k^{d+d'-3})\\\\
=\left[h^0(X, \mathcal{O}_X(kD))h^0(Y, \mathcal{O}_Y(kD))\right]
\left[\varphi_{X\times Y, k(D\boxtimes D')}(\lambda)\right]+O(k^{d+d'-3})\\\\
=\left[e_Xk^{d-1}+O(k^{d-2})\right]
\left[e_Yk^{d'-1}+ O(k^{d'-2})\right]
\left[\varphi_{X\times Y, k(D\boxtimes D')}(\lambda)\right]+O(k^{d+d'-3})\\\\
= \left[\varphi_{X\times Y, k(D\boxtimes D')}(\lambda)\right]
\left[e_Xe_Yk^{d+d'-2}\right]
+O(k^{d+d'-3}),\end{array}$$
Hence for any $x\geq 0$ and for any integer $k\geq 1$, we have
$$\frac{HKd(X\times Y, k(D\boxtimes D'))(x/k+1) }{k^{d+d'-2}}
= \left[\varphi_{X\times Y, D\boxtimes D'}(x)\right]
(e_Xe_Y) + O(1/k).$$
Now we fix $x\geq 0$ and  take lim as $k\to \infty$, then we have 
$$e_Xe_Y\left[\varphi_{X, D}(x)+ \varphi_{Y, D'}(x) - \varphi_{X, D}(x)
\varphi_{Y, D'}(x)\right]
= e_Xe_Y\left[\varphi_{X\times Y, D\boxtimes D'}(x)\right].$$
Therefore, for every $x\geq 0$, we have 
$$\varphi_{X, D}(x)+ \varphi_{Y, D'}(x) - \varphi_{X, D}(x)\varphi_{Y, D'}(x)
= \varphi_{X\times Y, D\boxtimes D'}(x).$$
This implies the proposition.\hfill$\Box$

\begin{defn}\label{d4}A rational polytope $P_D$  {\it tiles the space} $\R^{d-1}$
if for some $\lambda>0$
\begin{enumerate}
\item
$\bigcup_{v\in \Z^{d-1}} \left(v + P_{\lambda D}\right) = 
\R^{d-1}$
and 
\item $\dim \left[\left(v+ P_{\lambda D}\right) 
\cap \left(v'+ P_{\lambda D}\right)\right] < d-1$
if $v\neq v'$.
\end{enumerate}

Equivalently 
\begin{enumerate}
\item $\bigcup_{v\in \Z^{d-1}} \left[(v, 1)+ C_D \right]\cap \{z= \lambda +1\} = 
\R^{d-1}\times \{z= \lambda +1\}$
and 
\item $\dim \left[(v, 1)+ C_D \right] 
\cap \left[(v', 1)+ C_D \right]\cap \{z= \lambda +1\} 
< d-1$
if $v\neq v'$.
\end{enumerate}
It follows from the definition that if $P_D$ tiles the space 
$\R^{d-1}$ at $\lambda $ then $\lambda = (\mbox{Vol}_{d-1}(P_D))^{1-d}$.

In the literature this is known as {\em a simple tiling} (or $1$-tiling) 
 by the polytope $P_D$ with the lattice $M = \Z^{d-1}$.
\end{defn}

\begin{thm}\label{t4}
Let $(X, D)$ be  a toric pair of dimension $d-1\geq 1$. Then 
 $$ \left({e_0(X,D)}\right)^{\frac{2-d}{d-1}} \lim_{k\to \infty} 
\frac{e_{HK}(X, kD) - {e_0(X, kD)}/{d!}}{k^{d-2}}
\geq  \left[\frac{d-1}{d}\right]\left[(d-1)!\right]^{\frac{2-d}{d-1}}.$$
Morever, 
the equality hold, {\it i.e.}, 
\begin{equation}\label{e4}\left({e_0(X,D)}\right)^{\frac{2-d}{d-1}}
\lim_{k\to \infty} \frac{e_{HK}(X, kD) - 
{e_0(X, kD)}/{d!}}{k^{d-2}} = 
\left[\frac{d-1}{d}\right]\left[(d-1)!\right]^{\frac{2-d}{d-1}}.\end{equation}
if and only if $P_{D}$ tiles the space $\R^{d-1}$ for some $\lambda >0$.
\end{thm}
\begin{proof} Let  $(X, D)$ be a toric pair  of dimension $d-1$.
 We choose 
a real number $\alpha >0 $ such that $e_0(X,D)= \alpha^{d-1}(d-1)!$.

For $v\in \Z^{d-1}$ and $\lambda \geq 0$, let
 $$P_{\lambda D}^v = (P_{\lambda D}\times \{z=\lambda +1\})\cap 
(W_{v}\times\{z=\lambda+1\}).$$
Note that 
$$P_{\lambda D}\times \{z=\lambda +1\} = ((\underline{0}, 1)+C_D)\cap \{z=\lambda+1\}
= ((\underline{0}, 1)+C_D)\cap (\R^{d-1}\times \{z=\lambda+1\})$$
and 
$$P_{\lambda D}^v =  ((\underline{0}, 1)+C_D)\cap (W_{v}\times\{z=\lambda+1\}).$$
Hence 
\begin{equation}\label{wy2}
P_{\lambda D}\times \{z=\lambda +1\}=\bigcup_{v\in \Z^{d-1}}P_{\lambda D}^v.
\end{equation}
Also 
\begin{equation}\label{star2} P_{\lambda D}^v -(v, 0) =  ((-v, 1)+C_D) \cap 
(W_{\underline{0}}\times\{z=\lambda+1\}).\end{equation}
Therefore
\begin{equation}\label{star1}\bigcup_{u\in \Z^{d-1}}((u,1)+C_D)\cap (W_{\underline{0}}\times \{z= \lambda +1\}) = 
\bigcup_{v\in \Z^{d-1}}P_{\lambda D}^v-(v, 0),\end{equation}
where by (\ref{wy2}), 
$$\mbox{Vol}_{d-1}(\cup_{v\in \Z^{d-1}} P^v_{\lambda D}-(v, 0)) \leq 
\sum_{v\in \Z^{d-1}}\mbox{Vol}_{d-1}(P^v_{\lambda D}) = 
\mbox{Vol}_{d-1}(P_{\lambda D}).$$
Hence 
\begin{equation}\label{ptl}
\mbox{Vol}_{d-1}(W_{\underline{0}}\times\{z=\lambda+1\})\setminus 
\cup_{u\in \Z^{d-1}}((u,1)+C_D)\geq 1-\mbox{Vol}_{d-1}(P_{\lambda D}) =
1-\lambda^{d-1}\mbox{Vol}_{d-1}(P_D).\end{equation}
Therefore
$$\int_0^{\infty}\varphi_{X, D}(\lambda)d\lambda \geq
\int_0^{1/\alpha}\varphi_{X, D}(\lambda)d\lambda  
\geq \int_0^{1/\alpha}(1-\lambda^{d-1}\alpha^{d-1})d\lambda 
= \frac{1}{\alpha}
\int_0^1(1-\beta^{d-1})d\beta.$$
This implies
\begin{equation}\label{wy3}
\frac{e_0(X, D)}{(d-1)!}
\int_0^\infty\varphi_{X, D}(\lambda)d\lambda \geq  
\alpha^{d-2} \int_0^1(1-\beta^{d-1})d\beta.\end{equation}
If we denote 
$$A(X,D) = \lim_{k\to \infty} \frac{e_{HK}(X, kD) - {e_0(X, kD)}/{d!}}{k^{d-2}}$$ 
then we have 
$$\left({e_0(X,D)}\right)^{\frac{2-d}{d-1}}A(X, D) \geq  
\frac{((d-1)!)^{(2-d)/(d-1)}}{\alpha^{d-2}}\alpha^{d-2} \int_0^1(1-\beta^{d-1})d\beta
= \left[\frac{d-1}{d}\right]\left[(d-1)!\right]^{\frac{2-d}{d-1}}.$$
This proves Assertion~(1). 

\noindent~(2)\quad Suppose the polytope $P_D$ tiles the space 
$\R^{d-1}$, for some $\lambda_0>0$. 
Then, by (\ref{star1}) and Definition~\ref{d4}~(1), 
$$\bigcup_{v\in \Z^{d-1}}P_{\lambda_0 D}^v-(v, 0) = 
\bigcup_{v\in \Z^{d-1}}((-v,1)+C_D)\cap (W_{\underline{0}}\times 
\{z= \lambda_0 +1\}) = W_{\underline 0}\times \{z=\lambda_0+1\}.$$
This implies, by (\ref{star2}) and Definition~\ref{d4}~(2),
$$1 = \mbox{Vol}_{d-1}(\cup_{v\in \Z^{d-1}} P^v_{\lambda_0 D}-(v, 0)) =  
\sum_{v\in \Z^{d-1}}\mbox{Vol}_{d-1}(P^v_{\lambda_0 D}-(v, 0)) $$
$$ = \sum_{v\in \Z^{d-1}}\mbox{Vol}_{d-1}(P^v_{\lambda_0 D}) = 
\mbox{Vol}_{d-1}(P_{\lambda_0 D}) = \lambda_0^{d-1}\alpha^{d-1}.$$
This implies $\lambda_0 = 1/\alpha$. If $\lambda< \lambda_0$ then 
$\dim([P^v_{\lambda D}-(v,0)]\cap [P^{v'}_{\lambda D}-(v',0)]\cap 
\{z = \lambda+1\} <d-1$ implies  
$$\mbox{Vol}_{d-1}(\cup_{v\in \Z^{d-1}} P^v_{\lambda_0 D}-(v, 0)) = 
\mbox{Vol}_{d-1}(P_{\lambda D}) = \lambda^{d-1}\alpha^{d-1}, $$

Therefore
$$\int_0^{\infty}\varphi_{X, D}(\lambda)d\lambda = 
\int_0^{1/\alpha}\varphi_{X, D}(\lambda)d\lambda  
= \int_0^{1/\alpha}(1-\lambda^{d-1}\alpha^{d-1})d\lambda 
= \frac{1}{\alpha}
\int_0^1(1-\beta^{d-1})d\beta.$$
This implies
\begin{equation}\label{wy4}
\frac{e_0(X, D)}{(d-1)!}
\int_0^\infty\varphi_{X, D}(\lambda)d\lambda =   
\alpha^{d-2} \int_0^1(1-\beta^{d-1})d\beta.\end{equation}
Now the equality, as given in (\ref{e4}), follows from Theorem~\ref{t1}.

Conversely suppose the equality in (\ref{e4}) holds then 
retracing the above argument we have
$$\int_0^{\infty}\varphi_{X, D}(\lambda)d\lambda = 
\int_0^{1/\alpha}\varphi_{X, D}(\lambda)d\lambda  
+\int_{1/\alpha}^{\infty}\varphi_{X, D}(\lambda)d\lambda  
= \int_0^{1/\alpha}(1-\lambda^{d-1}\alpha^{d-1})d\lambda. $$
But, by (\ref{ptl}), for every $\lambda >0$, we have  
$\varphi_{X, D}(\lambda)\geq 1-\lambda^{d-1}\alpha^{d-1}$ 
and $\varphi_{X, D}(\lambda) \geq 0$. Hence the continuity of $\varphi_{X,D}$ 
(see Lemma~\ref{l4}) 
implies 
$$\begin{array}{lcl}
\varphi_{X, D}(\lambda) & = & 1-\lambda^{d-1}\alpha^{d-1} = 
1 - \mbox{Vol}_{d-1}(P_{\lambda D})~~\mbox{if}~~\lambda \leq 1/\alpha\\\
& = & 0~~\mbox{if}~~\lambda \geq \alpha.
\end{array}$$
This implies, for $\lambda_0 = 1/\alpha $, we have 
$$1 = \mbox{Vol}_{d-1}[\bigcup_{v\in \Z^{d-1}} (P^v_{\lambda_0 D}-(v, 0))] \leq 
 \sum_{v\in \Z^{d-1}}\mbox{Vol}_{d-1}(P^v_{\lambda_0 D})  = 
\mbox{Vol}_{d-1}(P_{\lambda_0 D}) = 1. $$
Therefore 
$$ \dim\left[(P^v_{\lambda_0 D}-
(v, 0))\cap (P^{v'}_{\lambda_0 D}-(v', 0))\right] <d-1$$
and 
$$\bigcup_{u\in \Z^{d-1}}((u,1)+C_D)\cap (W_{\underline{0}}\times 
\{z= \lambda_0 +1\}) = \bigcup_{v\in \Z^{d-1}}\left[P_{\lambda_0 D}^v-(v, 0)
\right]  
 = W_{\underline 0}\times \{z=\lambda_0+1\}.$$
Now, by Lemma~\ref{l2}, we can conclude the same thing, by 
replacing $W_{\underline 0}$ by $W_v$, for any $v\in \Z^{d-1}$. In particular 
$P_D$ tiles the space $\R^{d-1}$ for $\lambda_0 = 1/\alpha$.
This completes the proof of  Assertion~(2) and hence the theorem.
\end{proof}

\begin{ex} Let $(X_0, D_0)$ be the Segre self-product of 
$(\P^1, \sO_{\P^1}(m_0))$, taken $d-1$ times, {\it i.e.},  
$$(X_0, D_0) = (\P^1\times\cdots \times\P^1,
\sO_{\P^1}(m_0)\boxtimes\cdots \boxtimes\sO_{\P^1}(m_0)), $$
for some integer $m_0\geq 1$. Then
 the polytope $P_{D_0} = [0, m_0]^{d-1}$ and 
$(1/m_0)P_{D_0} = [0, 1]^{d-1}$.  This implies that  $P_{D_0}$ 
tiles the space $\R^{d-1}$ for 
$\lambda = 1/m_0$. 
\end{ex}

\begin{rmk}\label{r4}We recall the following 
 conjecture of Watanabe-Yoshida (Conjecture 4.2, [WY2]):
For a  Noetherian unmixed nonregular local ring $(R, {\bf m}, K)$ of dimension $d$ 
with $K = {\overline \F_p}$,
$$e_{HK}(R, {\bf m})\geq  e_{HK}(A_{p,d}, (Y_0, \ldots, Y_d)),$$ 
where $A_{p,d}$ is given by 
$A_{p,d} :=F_p[[Y_0,Y_1,...,Y_d]]/(Y_0^2 +\cdots+Y_d^2).$

\vspace{5pt}

Here  Theorem~\ref{t4}
implies that for any toric pair 
$(X, D)$ of dimension $d-1$, 
the asymptotic growth of the HK multiplicity (relative  to its  
usual multiplicity, $e_0(X, D)$)  is  always $\geq $  
 the asymptotic growth of the HK multiplicity (relative to 
its usual multiplicity $e_0(X_0, D_0)$) for the pair $(X_0, D_0)$, where 
$(X_0, D_0)$ is the Segre self-product, of any  toric pair of the type
$(\P^1, \sO_{\P^1}(m_0))$, taken $d-1$ times.
Note that 
the associated coordinate ring for any such pair $(X_0, D_0)$ 
is given by a set of  quadratic binomials over $K$.
\end{rmk}

\begin{rmk}\label{pt}Let $P$ be a rational convex polytope in $\R^{d-1}$; then 
we can formulate  the property that 
$P$ {\em tiles the space} $\R^{d-1}$ in terms of 
HK multiplicity, as follows: 

We  choose $m>>0$ (by Corollary~2.2.18 in [CLS], any
$m \geq (d-2)n_1$, where $n_1P$ is an integral polytope) 
such that 
$mP$ is a very ample integral convex polytope. In particular there is a toric pair 
$(X, D)$ such that the associated polyope $P_D = mP$.
Then the polytope $P$ tiles the space $\R^{d-1}$ for some $\lambda >0$ if and only if
$$\left({e_0(X,D)}\right)^{\frac{2-d}{d-1}}
\lim_{k\to \infty} \frac{e_{HK}(X, kD) - 
{e_0(X, kD)}/{d!}}{k^{d-2}} = 
\left[\frac{d-1}{d}\right]\left[(d-1)!\right]^{\frac{2-d}{d-1}}.$$

Note that this criteria is independent of the choice of $m$, as left hand side of
the above equation deos not change if we replace $D$ by an integral multiple of $D$.
Moreover, if $P$ tiles the space $\R^{d-1}$ then it tiles at $\lambda = 
(\mbox{Vol}~P)^{1-d}$.
\end{rmk}

\section{Examples}

\begin{ex}\label{ep1}
We compute the $HK$ density function for the toric 
pair $(X, D)=(\P^1, \sO(n))$ for $n\in \N$. 
The polytope $P_D$ can be taken to be the line segment $[0, n]$ 
(upto translation by integer points). Then $\sP_D=\bigcup_{i=0}^{n-1} P_i$ 
where $P_i=\text{Conv}\left((0,0), (i,1), (i+1,1), (i+1, \frac{n+1}{n})
\right)$, $i=0,\ldots, n-1$. One has 
\begin{eqnarray*}
HKd(X, D)(\lambda)&=&
\begin{cases} 
                                           n\lambda & \text{ if $ 0\leq \lambda < 1$} \\
                                          n(1-n (\lambda-1))  &  \text{ if $ 1\leq \lambda < 1+\frac{1}{n}$}\\ 
                                          0 & \text{ if $\lambda\geq 1+\frac{1}{n}.$} 
                                          \end{cases}
\end{eqnarray*}
Moreover $\varphi_{kD}(\lambda) = 1-nk\lambda$ if $0\leq \lambda <1/nk$ and 
$\varphi_{kD}(\lambda) = 0$ otherwise.
\end{ex}
\begin{ex}
\label{e1}
We compute the HK density function for the Hirzebruch surface $X=F_a$ (See [T1] for a different 
geometric approach for this) with 
parameter $a\in{\mathbb{N}}$, which is a ruled surface over $\P^1_k$, 
where k is a field of characteristic $p> 0$. See [Fu] for a detailed description 
of the surface as a toric variety. The $T$-Cartier divisors are given by 
$D_i=V(v_i), i=1,2 ,3 , 4$, where $v_1=e_1, v_2=e_2, v_3=-e_1+ae_2, v_4=-e_2$ and 
$V(v_i)$  denotes the $T$-orbit closure corresponding to the cone generated by $v_i$. 
We know the Picard group is generated by $\{D_i\ : i=1, 2,3, 4\}$ over $\mathbb{Z}$. 
One can check the only relations in $\text{Pic}(X)$\ can be described by 
$D_3\thicksim D_1$ and $D_2\thicksim D_4-aD_1$. Therefore 
$\text{Pic}(X)=\mathbb{Z}D_1\oplus\mathbb{Z}D_4$. One can use standard methods in toric geometry to see 
that $D=cD_1+dD_4$ is ample  if and only if $a, c>0$. Then 
$P_D=\{(x,y)\in M_\mathbb{R}\ |\ x\geq -c, y\leq d, x\leq ay\}$ and 
$\alpha^2=\text{Vol}(P_D)=cd+\frac{ad^2}{2}$. To consider $HKd(X,D)$  for 
$D=cD_1+dD_4$, we split it into two different cases.

\begin{enumerate}
\item Case 1: $c \geq d $ 
\begin{eqnarray*}
                                   HKd(X,D)(\lambda)&=&    
\begin{cases} 
                                                      (cd+\frac{ad^2}{2})\lambda^2 & \text{ if 
$ 0\leq \lambda < 1$} \\\\
                                                      (cd+\frac{ad^2}{2})\lambda^2 \\
                                                      - (c+\frac{ad}{2}+1)(d+1)(cd+\frac{ad^2}{2})(\lambda-1)^2  &  \text{ if $ 1\leq \lambda < 1+\frac{1}{c+ad}$}\\\\
                                                      (c+\frac{ad}{2})(d+1)\frac{1}{2a}(c+1-c\lambda)^2\\
                                                      +(cd+\frac{ad^2}{2})\lambda(d+1-d\lambda) &  \text{ if $ 1+\frac{1}{c+ad}\leq \lambda < 1 +\frac{1}{c} $}\\\\
                                                      (cd+\frac{ad^2}{2})\lambda(d+1-d\lambda)   & \text{ if $1+\frac{1}{c}\leq \lambda < 1 +\frac{1}{d}.$}\\
                                                      0 & \text{ if $\lambda\geq 1+\frac{1}{d}.$} 
                                          \end{cases}
\end{eqnarray*}

\item Case 2: $c \leq d $ 
\begin{eqnarray*}
                                   HKd(X,D)(\lambda)&=&    
\begin{cases} 
                                                      (cd+\frac{ad^2}{2})\lambda^2 & \text{ if $ 0\leq \lambda < 1$} \\\\
                                                      (cd+\frac{ad^2}{2})\lambda^2 \\
                                                      - (c+\frac{ad}{2}+1)(d+1)(cd+\frac{ad^2}{2})(\lambda-1)^2  &  \text{ if $ 1\leq \lambda < 1+\frac{1}{c+ad}$}\\\\
                                                      (c+\frac{ad}{2})(d+1)\frac{1}{2a}(c+1-c\lambda)^2\\
                                                       +(cd+\frac{ad^2}{2})(d+1-d\lambda) &  \text{ if $ 1+\frac{1}{c+ad}\leq \lambda < 1 +\frac{1}{d}$}\\\\
                                                      (cd+\frac{ad^2}{2} +\frac{ad}{2})\frac{1}{2a}\big(a+1-(c+ad)(\lambda-1)\big)^2\\
                                                      +\frac{c}{2a}(c+1-c\lambda)^2 & \text{ if $1+ \frac{1}{d}\leq \lambda < 1 +\frac{a+1}{ad+c}$}\\\\
                                                        \frac{c}{2a}(c+1-c\lambda)^2&  \text{ if $1 +\frac{a+1}{ad+c}\leq \lambda < \frac{1}{c}$.}\\
                                                        0 & \text{ if $\lambda\geq 1+\frac{1}{c}.$} 
                                                      \end{cases}
\end{eqnarray*}

\end{enumerate}
\end{ex}
\begin{ex}
\label{ex4.8}
In this example we consider how  $\varphi_{X,D}$ changes as $D$ varies over 
the ample cone of divisors on $X$. We consider this question for the Hirzebruch 
surface $X=F_a$ with parameter $a\in{\mathbb{N}}$,  as in Example \ref{e1}. 
Let $D=cD_1+dD_4$ be a very ample $T$-Cartier divisor. 
Then $P_D=\{(x,y)\in M_\mathbb{R}\ |\ x\geq -c, y\leq d, x\leq ay\}$. 
To consider $\varphi_{X,D}$ for $D=cD_1+dD_4$, we split it into two different cases.
\begin{enumerate}
 \item  When $c\geq d$\ :
  \begin{eqnarray*}
  \varphi_{X,D}(\lambda)&=&
  \begin{cases}
  1-\lambda^2(cd +\frac{ad^2}{2}) & \text{if $0\leq \lambda \leq \frac{1}{ad+c}$,}\\
  (1-\lambda d)+ \frac{(1-\lambda c)^2}{2a} & \text{if $\frac{1}{ad+c}\leq \lambda \leq \frac{1}{c}$,}\\
  1-\lambda d & \text{if $\frac{1}{c}\leq \lambda \leq \frac{1}{d}$,}\\
  0 & \text{if $\lambda \geq \frac{1}{d}$.}
  \end{cases}
 \end{eqnarray*}
 \item When $c\leq d$\ :
  \begin{eqnarray*}
  \varphi_{X,D}(\lambda)&=&
  \begin{cases}
  1-\lambda^2(cd +\frac{ad^2}{2}) & \text{if $0\leq \lambda \leq \frac{1}{ad+c}$,}\\
  (1-\lambda d)+ \frac{(1-\lambda c)^2}{2a} & \text{if  $\frac{1}{ad+c}\leq \lambda \leq \frac{1}{d}$,}\\
   \frac{(1+a-\lambda(ad+ c))^2}{2a} & \text{if $\frac{1}{d}\leq \lambda \leq \frac{1+a}{ad+c}$,}\\
  0 & \text{if $\lambda\geq  \frac{1+a}{ad+c}$.}
  \end{cases}
 \end{eqnarray*}
\end{enumerate}
\end{ex}

\begin{ex}
Here we compute the $\varphi_{X, -K}$ of the smooth Fano toric 
varieties $X$ of dimension $d-1=2$ with respect to the anticanonical 
divisor $-K$, namely $\P^2$, and blow ups of $\P^2$ at one, 
two and three points with respect to the anticanonical divisor $-K=\sum D_i$, 
where $D_i$ are the T-Cartier divisors on the respective varieties. 
We find $P_{-K}$, and eventually $\varphi_{X, -K}$. $\varphi_{X, -K}$ equals 
the volume of the darker shaded region at $Z=\lambda$. For each surface we 
denote the co-ordinate ring by $R$ and the homogeneous maximal ideal by 
${\bf m}$ with respect to the respective embedding.
Let $$A(X,D) = \lim_{k\to \infty} \frac{e_{HK}(R, {\bf m}^k) - 
e_0(R,{\bf m}^k)/d!}{k^{d-1}}.$$
\begin{enumerate}
\item {$\P^1\times \P^1$}
\begin{figure}[h]
\begin{tabular}{ccc}
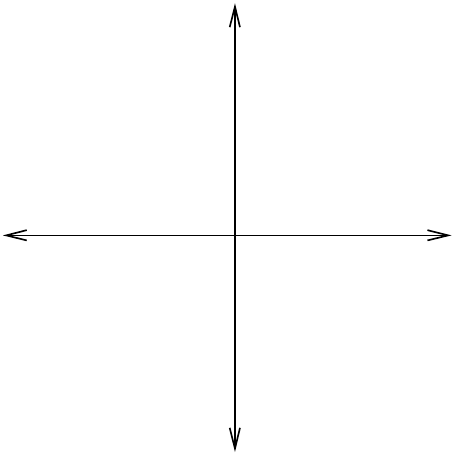 &\hspace{.2 cm} 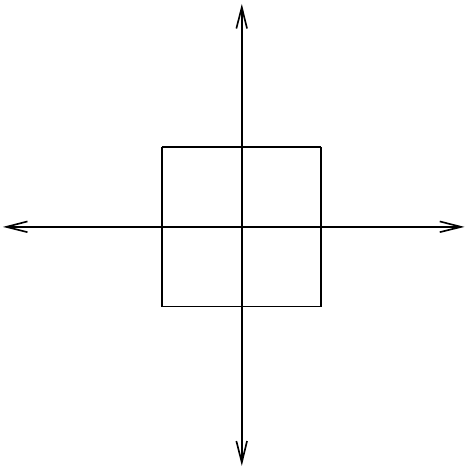 & \hspace{.2cm}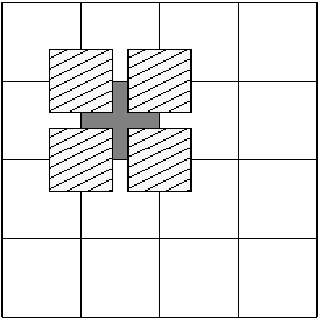\\
$\P^1\times\P^1$&$P_{-K}$ & \hspace{.2cm}$\varphi_{\P^1\times\P^1, -K}(\lambda)$ at $\lambda=1/\sqrt{5}$\\
\end{tabular}
\label{Figure 1}
\caption{}
\end{figure}
\begin{eqnarray*}
\varphi_{\P^1\times\P^1, -K}(\lambda)&=&
\begin{cases}
1-4\lambda^2 & \text{if $0\leq \ \lambda\leq \frac{1}{2}$,}\\
0 & \text{otherwise},
\end{cases}\\
&&\hspace{-2.7cm}
\int \varphi_{\P^1\times\P^1, -K}(\lambda)d\lambda =\frac{1}{3}~~\mbox{and}~~~
A(\P^1\times\P^1, -K) = 2\left(\frac{1}{3}\right) = \frac{2}{3}.
\end{eqnarray*}
\item {$\P^2$, the projective space}
\begin{figure}[h]
\begin{tabular}{ccc}
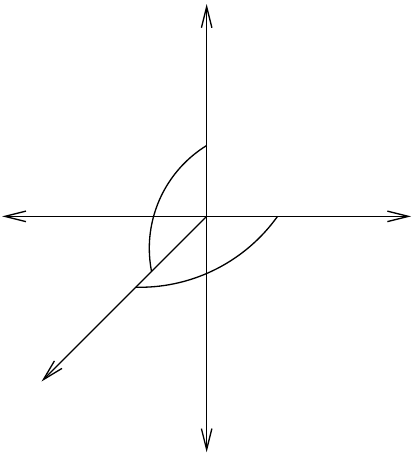 &\hspace{.4 cm} 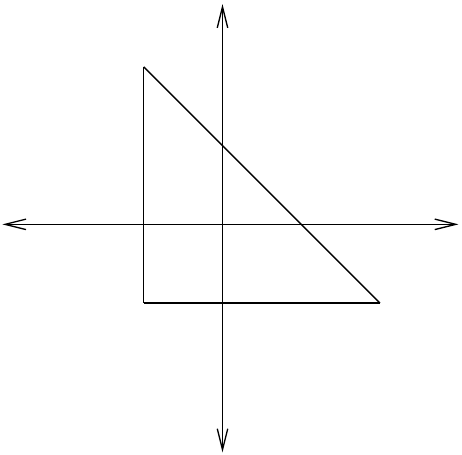 & \hspace{.7cm}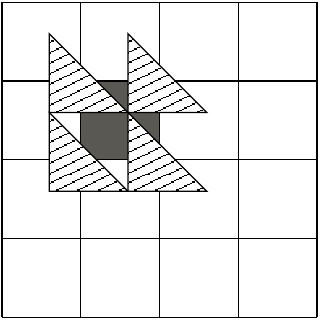\\
$\P^2$&$P_{-K}$ & \hspace{.6cm}$\varphi_{\P^2, -K}(\lambda)$ at $\lambda=1/3$\\
\end{tabular}
\label{Figure 1}
\caption{}
\end{figure}
\begin{eqnarray*}
\varphi_{\P^2, -K}(\lambda)&=&
\begin{cases}
1-\frac{9}{2}\lambda^2 & \text{if $0\leq \ \lambda\leq \frac{1}{3}$,}\\
\frac{1}{2}(2-3\lambda)^2 & \text{if $\frac{1}{3}\leq \ \lambda\leq \frac{2}{3}$,}\\
0 & \text{otherwise},
\end{cases}\\
&&\hspace{-2.2cm }
\int \varphi_{\P^2, -K}(\lambda)d\lambda=
\frac{1}{3}~~\mbox{and}~~~A(\P^2, -K) = \left(\frac{9}{4}\right)\left(\frac{1}{3}\right) = \frac{3}{4}.
\end{eqnarray*}
\item{$X_3=$ blow-up of $\P^2$ at one point}
\begin{figure}[h]
\begin{tabular}{ccc}
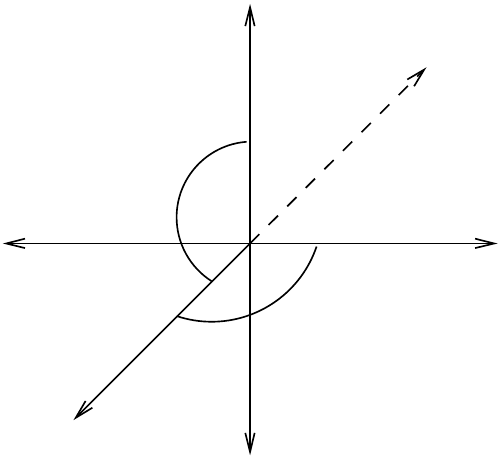 &\hspace{.2cm} 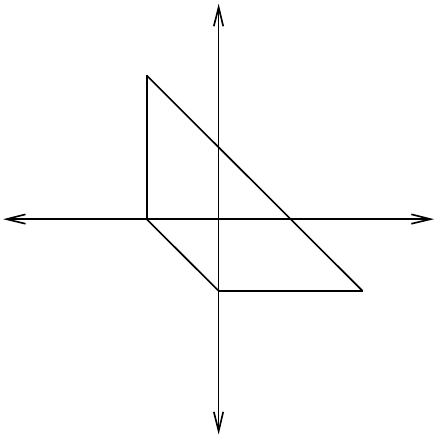 & \hspace{.4cm}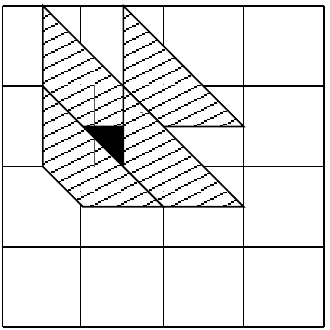\\
$X_3=$ blow-up of $\P^2$ at one point & $P_{-K}$ & \hspace{.4cm}$\varphi_{X_3, -K}(\lambda)$ at $\lambda=1/2$\\
\end{tabular}
\label{Figure 2}
\caption{}
\end{figure}
\begin{eqnarray*}
\varphi_{X_3, -K}(\lambda)&=&
\begin{cases}
1-4\lambda^2 & \text{if $0\leq \ \lambda\leq \frac{1}{3}$,}\\
\frac{1}{2}(\lambda^2-6\lambda+3)&\text{if $\frac{1}{3}\leq \ \lambda\leq \frac{1}{2}$,}\\
\frac{1}{2}(2-3\lambda)^2 & \text{if  $\frac{1}{2}\leq \ \lambda\leq \frac{2}{3},$}\\
0 & \text{otherwise},
\end{cases}\\
&&\hspace{-2.2cm }\int \varphi_{X_3, -K}(\lambda)d\lambda=
\frac{25}{72}~~\mbox{and}~~A(X_3, -K) = 2\left(\frac{25}{72}\right) = \frac{25}{36}.
\end{eqnarray*}
 \item {$X_4=$ blow-up of $\P^2$ at two points}
 \begin{figure}[h]
\begin{tabular}{ccc}
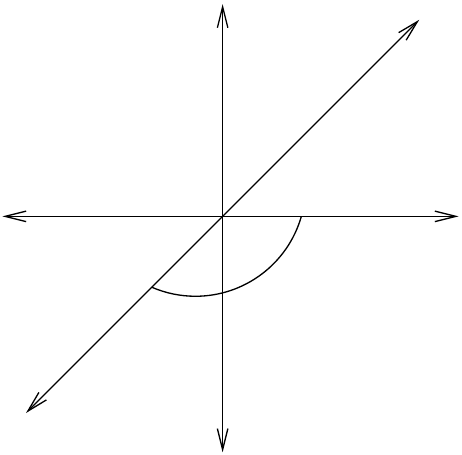 &\hspace{.1 cm} 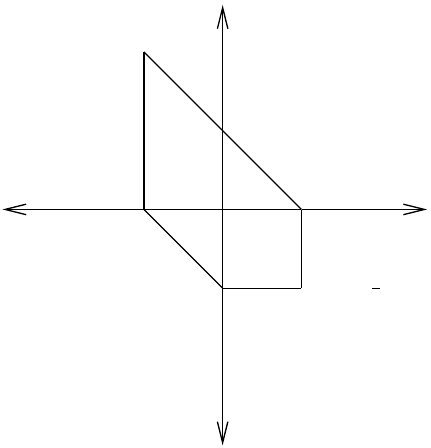 & \hspace{.6cm}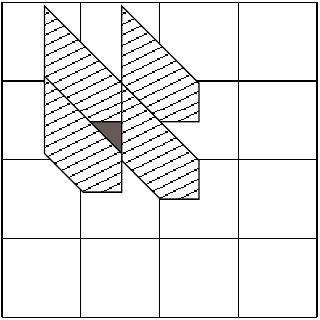\\
$X_4=$ blow-up of $\P^2$ at two points & $P_{-K}$ &\hspace{.4cm} $\varphi_{X_3, -K}(\lambda)$ at $\lambda=1/2$\\
\end{tabular}
\label{Figure 3}
\caption{}
\end{figure}
\begin{eqnarray*}
\varphi_{X_4,-K}(\lambda)&=&
\begin{cases}
1-\frac{7}{2}\lambda^2 & \text{if $0\leq \ \lambda\leq \frac{1}{2}$,}\\
\frac{1}{2}(2-3\lambda)^2 & \text{if $\frac{1}{2}\leq \ \lambda\leq \frac{2}{3}$,}\\
0 & \text{otherwise},
\end{cases}\\
&&\hspace{-2.2cm }
\int \varphi_{X_4, -K}(\lambda)d\lambda=  \frac{13}{36}~~\mbox{and}~~A(X_4, -K) =   
\left(\frac{7}{4}\right)\left(\frac{13}{36}\right) = \frac{91}{144}.
\end{eqnarray*}
\item {$X_5=$ blow-up of $\P^2$ at three points}
\begin{figure}[h]
\begin{tabular}{ccc}
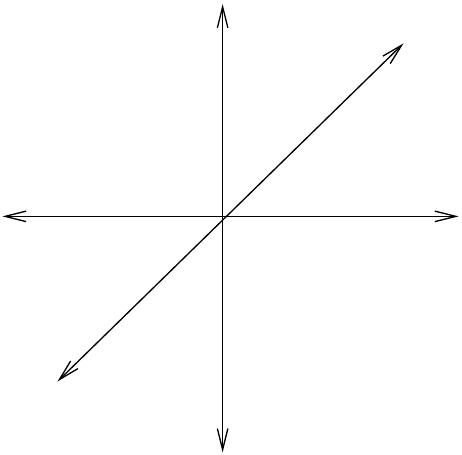 &\hspace{.2cm} 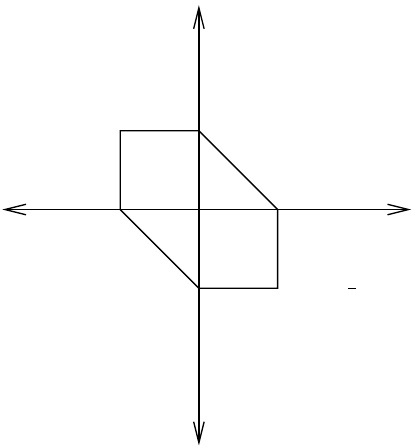 & \hspace{.5cm}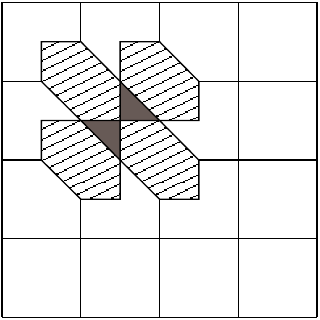\\
$X_5=$ blow-up of $\P^2$ at three points & $P_{-K}$ & \hspace{.4cm}$\varphi_{X_4, -K}(\lambda)$ at $\lambda=1/2$\\
\end{tabular}
\label{Figure 4}
\caption{}
\end{figure}
\begin{eqnarray*}
\varphi_{X_5,-K}(\lambda)&=&
\begin{cases}
1-3\lambda^2 & \text{if $0\leq \ \lambda\leq \frac{1}{2}$,}\\
(2-3\lambda)^2 & \text{if $\frac{1}{2}\leq \ \lambda\leq \frac{2}{3}$,}\\
0 & \text{otherwise},
\end{cases}\\
&&\hspace{-2.2cm }\int \varphi_{X_5, -K}(\lambda)d\lambda= \frac{7}{18} 
~~\mbox{and}~~A(X_5, -K) = \left(\frac{3}{2}\right)\left(\frac{7}{18}\right) 
= \frac{7}{12}.
\end{eqnarray*}
\end{enumerate}
\end{ex}
\newpage
\begin{rmk}Given $A(X,D)$ 
as in the above example, if
  we define (see Theorem~\ref{t4} and (\ref{e4})) 
$$B(X,D) = \left({e_0(X,D)}\right)^{\frac{2-d}{d-1}}A(X,D) - 
\left[\frac{d-1}{d}\right]\left[(d-1)!\right]^{\frac{2-d}{d-1}},$$
then we have 
$$B(\P^1\times \P^1, -K) = 0~~\mbox{and}~~~B(\P^2, -K) > 
B(X_3, -K) >  B(X_4, -K) > B(X_5, -K).$$
\end{rmk}

\begin{rmk} The equality given by (\ref{e4}) can be achieved by a toric pair
$(X, D)$ other than a self product of $(\P^1, \sO(m_0))$.
However, for $d-1=2$ and such a pair $(X, D)$, $P_D$
must be a centrally symmetric hexagon (a convex body 
$C\subset \R^{d-1}$ is said to be centrally symmetric with respect to origin, 
if $x\in C$ if and only if $-x\in C$), see [Sc].  
\begin{figure}[h]
\begin{tabular}{cc}
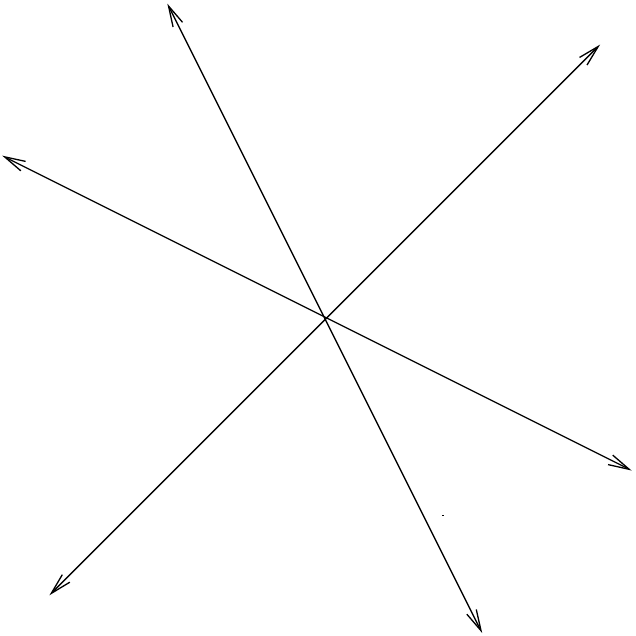 &\hspace{.2cm} 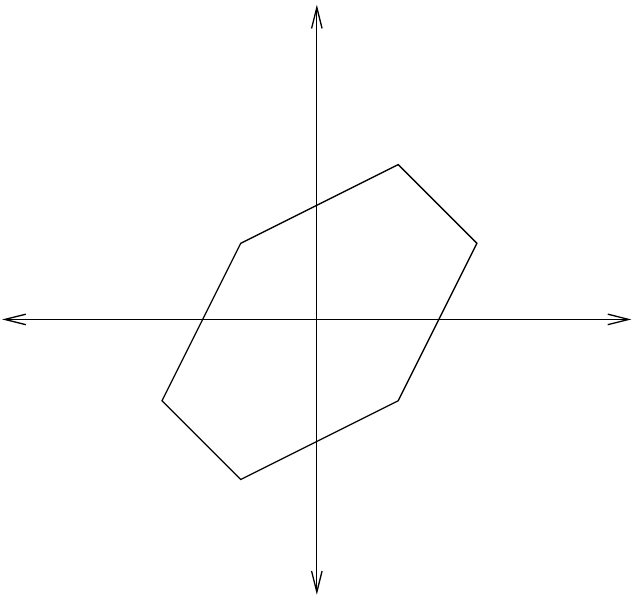 \\
& $P_{D}, D=3\sum D_i$\\\\
\end{tabular}
\caption{}
\end{figure}
For $d=3$, consider the fan $\Delta$ in Figure 6, where 
$v_1=2e_1-e_2, v_2=e_1+e_2, v_3=-e_1+2e_2, v_4=-v_1, v_5=-v_2, v_6=-v_3$. The fan $\Delta$ 
has the maximal cones $\sigma_i= \langle v_i, v_{i+1}\rangle, i=1, \ldots, 6$, with the convention 
$v_7=v_1$. This gives a singular toric surface, since the cones $\sigma_i$ s are not smooth 
(Theorem 3.1.18, [CLS]). Consider the divisor $D=3\sum_i D_i$ to 
get $P_D$ as in Figure 6. Since dimension of $P_D$ is $2$, $P_D$ is normal (Theorem 2.2.12, [CLS]) 
and hence is very ample (Proposition 2.2.18, [CLS]). By Proposition 6.1.10, [CLS] it follows 
that $D$ is very ample. We see that such a $P_D$ is indeed possible, where $D$ is a very ample $T$-Cartier 
divisor on $X(\Delta)$ with Vol($P_D)=\alpha^2$ such that $\varphi_{X,D}({\lambda}/{\alpha})=1-\lambda^2$.\\
\begin{figure}[h]
\begin{tabular}{cc}
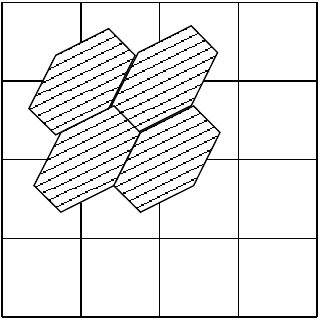 &\hspace{2cm} 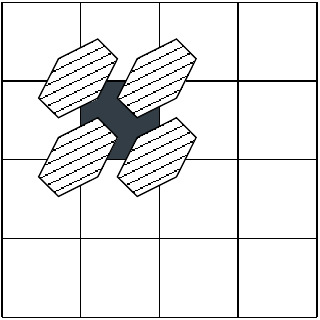 \\
$\varphi_D(\lambda)=0 \text{ at } \lambda=1/3$ &\hspace{2cm}$\varphi_D(\frac{\lambda}{3})=1-\lambda^2$\\
\end{tabular}
\caption{}
\end{figure}
\end{rmk}

\end{document}